[1994/12/01]
\documentclass[10pt, reqno]{amsart}
\usepackage[english, activeacute]{babel}
\usepackage{amsmath,amsthm,amsxtra}
\usepackage{epsfig}
\usepackage{amssymb}
\usepackage{latexsym}
\usepackage{amsfonts}
\usepackage{centernot}
\usepackage{mathtools}
\usepackage{ stmaryrd }
\usepackage{hyperref}
\usepackage{pdftricks}
\usepackage{wrapfig}
\usepackage{tikz}
\usetikzlibrary{arrows,calc,decorations.pathreplacing}
\definecolor{light-gray1}{gray}{0.90}
\definecolor{light-gray2}{gray}{0.80}
\definecolor{light-gray3}{gray}{0.60}

\title{The Calder\'on problem for quasilinear elliptic equations}
\author{Claudio Mu\~noz}
\author{Gunther Uhlmann}
\address{CNRS and Departamento de Ingenier\'ia Matem\'atica DIM, Universidad de Chile, Chile}
\email{cmunoz@dim.uchile.cl, claudio.munoz@math.u-psud.fr}
\address{Department of Mathematics, University of Washington, Box 354350
Seattle, Washington 98195, USA\footnote{Also affiliated to Jockey Club Institute for Advanced Study, HKUST,
Acknowledgement
Clear Water Bay, Kowloon, Hong Kong, China.}}
\email{gunther@math.washington.edu}
\date{\today}
\thanks{}

\chardef\bslash=`\\ 

\newtheorem{thm}{Theorem}[section]
\newtheorem{cor}[thm]{Corollary}
\newtheorem{lem}[thm]{Lemma}
\newtheorem{prop}[thm]{Proposition}

\newtheorem{Cl}[thm]{Claim}
\newtheorem{defn}[thm]{Definition}

\theoremstyle{remark}

\newtheorem{rem}[thm]{Remark}

\numberwithin{equation}{section}

\newcommand{\R}{\mathbb{R}}

\newcommand{\Com}{\mathbb{C}}
\newcommand{\la}{\lambda}

\newcommand{\al}{\alpha}
\newcommand{\bt}{\beta}
\newcommand{\ga}{\gamma}

\newcommand{\re}{\operatorname{Re}}

\def\bm{\left( \begin{array}{cc}}
\def\endm{\end{array}\right)}

 \newcommand{\Dv}{\operatorname{div}}
 \providecommand{\abs}[1]{\lvert#1 \rvert}
 \providecommand{\norm}[1]{\lVert#1 \rVert}

\newcommand{\be}{\begin{equation}}
\newcommand{\ee}{\end{equation}}
\newcommand{\ba}{\left(\begin{array}{c}}
\newcommand{\ea}{\end{array}\right)}
\newcommand{\bea}{\begin{eqnarray}}
\newcommand{\eea}{\end{eqnarray}}
\newcommand{\bee}{\begin{eqnarray*}}
\newcommand{\eee}{\end{eqnarray*}}
\newcommand{\ben}{\begin{enumerate}}
\newcommand{\een}{\end{enumerate}}

\newcommand{\eval}[2][\right]{\relax
  \ifx#1\right\relax \left.\fi#2#1\rvert}

\let\abs=\envert

\let\norm=\enVert

\begin{document}
\begin{abstract}
In this paper we show uniqueness of the conductivity for the quasilinear Calder\'on's inverse problem. The nonlinear conductivity depends, in a nonlinear fashion, of the potential itself and its gradient. Under some structural assumptions on the direct problem, a real-valued conductivity allowing a small analytic continuation to the complex plane induce a unique Dirichlet-to-Neumann (DN) map. The method of proof considers some complex-valued, linear test functions based on a point of the boundary of the domain, and a linearization of the DN  map placed at these particular set of solutions. 
\end{abstract}
\maketitle \markboth{Quasilinear inverse problems} {Mu\~noz and Uhlmann}
\renewcommand{\sectionmark}[1]{}

\section{Introduction}

\medskip

\subsection{Setting of the problem}
Let $\Omega \subset \R^n$, $n\geq 2,$ be a smooth $C^{2,\al}$ bounded domain,  for some $0<\al<1$. Acting on $\Omega$, we will consider a nonlinear, uniformly (in $\Omega$) positive function 
\be\label{a_hyp}
\begin{aligned}
& a: \Omega \times \R \times \R^n \to (0,\infty), \\
& a=a(x,s,p) \geq a_0(s,p)>0, \quad \hbox{$a_0$ given.}
\end{aligned}
\ee
The purpose of this paper is to describe the Calder\'on's inverse problem for a \emph{quasilinear} conductivity $a(\cdot)$, that is to say, the study of the quasilinear scalar equation 
\be\label{IP}
\Dv_x \big[ a(x,u(x),\nabla u(x))~\nabla u(x) \big] =0, \quad x\in \Omega.
\ee
Here $u=u(x)$ is assumed to be a function defined from $\overline{\Omega}$ into $\R$. In order to determine a possibly unique $u$, we will impose a boundary condition
\[
u \big|_{\partial\Omega} =f,
\]
for some fixed $f$ in an space of smooth functions, to be specified below. 

\medskip

The standard and well-known Calder\'on's problem, namely the determination of the conductivity $a=a(x)$ for the problem
\be\label{Calderon}
\Dv_x [ a(x)\, \nabla u(x)] =0, \quad x\in \Omega, \qquad u \big|_{\partial\Omega} =f,
\ee
under the knowledge of the \emph{Dirichlet-to-Neumann map} (DN)
\[
H^{1/2}(\partial\Omega) \ni f  \longmapsto a ~\nabla u \cdot \nu \big|_{\partial\Omega} \in H^{-1/2}(\partial\Omega), \qquad \hbox{$\nu$ unit outer normal to $\Omega$,}
\]
has attracted the attention of many researchers during the past thirty years. Outstanding results in this area are the works by Calder\'on \cite{Cal}, Sylvester the second author  \cite{SU0,SU}, Nachman \cite{Nachman}, Astala and P\"aiv\"arinta \cite{AP}, among many others. The survey \cite{Gu_survey} is a suitable source for a historical account on the developments of the Calder\'on's problem.

\medskip

However, in nonlinear media applications (see \cite{Sun2} for a detailed survey), the conductivity $a(x)$ is usually a nonlinear function, not only depending on the point $x$ but also on the function $u(x)$, and more importantly, on its gradient $\nabla u(x)$. It is for this reason that problem \eqref{IP} is a natural step towards the understanding of several inverse problems coming from different applied scientific problems. 

\medskip

More precisely, by \emph{quasilinear inverse problem} associated to \eqref{IP}, we mean the following question: under which conditions on the conductivity $a$, the boundary values $f$, and a related Dirichlet-to-Neumann map associated to $f$, $a$ and $u$,  we can recover the coefficient
\[
a=a(x,s,p), \quad \hbox{where} \quad (x,s,p) \in\Omega\times \R \times \R^n.
\]
Note that we must recover a scalar function depending on $2n+1$ variables, where $n\geq 2$ is the dimension of space. This inverse problem for \eqref{IP} is in some sense hard to tackle down because of the gradient term $\nabla u$ inside the conductivity $a$, which makes the problem effectively quasilinear, and standard methods do not apply except for very particular situations where the coefficient $a$ has particular properties. For example, in the case where the coefficient $a$ does not depend on the gradient $\nabla u$, namely $a=a(x,s)$ only, Sun \cite{S} showed that the knowledge of the DN map
\be\label{DN_map_0}
C^{2,\al}(\partial\Omega) \ni f \longmapsto  \Gamma_a[f] := a(x,u) \nabla u \cdot \nu\big|_{\partial\Omega} \in C^{1,\al}(\partial\Omega),  
\ee
where $\nu$ is the outer normal in $\partial\Omega$, and  $u=u_f$ is solution of the equation
\[
\Dv (a(x,u)\nabla u) =0 ~ \hbox{ in }~ \Omega, \quad u\big|_{\partial\Omega}=f,
\]
determines the strictly positive coefficient $a(x,s)$. The fundamental step in their proof is to linearize the DN map \eqref{DN_map_0}, following the original idea of Isakov \cite{Isakov}. The objective is then to show that equality of quasilinear DN maps leads to a corresponding equality at the level of linearized DN maps, where a much better developed theory is available. Afterwards, Sun and the second author \cite{SunU} extended this result by considering \emph{anisotropic conductivities}: the DN map
\[
\nu \cdot(A(x,u)\nabla u)\big|_{\partial\Omega}
\]
determines the matrix-valued conductivity $A(x,s)$, in the case where $u$ solves the equation
\[
\Dv (A(x,u)\nabla u) =0 \hbox{ in }\Omega, \quad u\big|_{\partial\Omega}=f,
\]
and $A$ is a symmetric, positive definite matrix. Note that uniqueness is obtained up to a change of coordinates that leaves invariant the boundary: if $\Phi :\overline{\Omega} \to \overline{\Omega}$ is a $C^1$ diffeomorphism that satisfies $\Phi = \hbox{Id}$ on the boundary $\partial \Omega$, then
\be\label{Difeo}
A_\Phi(x,u):= |D\Phi|^{-1} D\Phi^T A(\Phi^{-1} x, u(\Phi^{-1} x))
\ee
is another conductivity that has the same DN map. Here $D\Phi$ is the Jacobian matrix of $\Phi$, and $|D\Phi|$ its Jacobian determinant. Later, Hervas and Sun \cite{HS} considered the problem for the quasilinear problem
\[
\Dv (A(x,u,\nabla u)) =0 \hbox{ in }\Omega, \quad u\big|_{\partial\Omega}=f,
\]
and where $A$ satisfy one of the following two conditions: either $A(x,u,\nabla u)=A_0(x)\nabla u$ (linear case on $\nabla u$, no dependence on $u$), or  $A(x,u,\nabla u)=A_0(\nabla u)$ (no dependence on $x$ nor $u$ at all). In both cases, uniqueness is obtained up to a diffeomorphism that changes coordinates, similar as in \eqref{Difeo}. Then, the natural question is the following: can on improve Sun-Uhlmann and Hervas-Sun results by allowing a complete quasilinear conductivity as in \eqref{IP}?

\medskip

A simple but somehow naive approach to this question should be to extend Hervas-Sun's result \cite{HS} by allowing the conductivity to depend on $x$, $u$ and $\nabla u$. However, one can easily deduce that the problem is in some formal sense \emph{undetermined}, because one has to recover a scalar function depending on $n+1+n=2n+1$ variables, and the corresponding DN map provides much less information. As far as we understand, this problem is completely open. A second issue comes from the fact that even the \emph{solvability theory} for the direct problem is not completely well-understood in the classical sense, and additional conditions are usually needed: either $(i)$ one has solvability for $u$ with small gradient and a few mild assumptions on the conductivity $a$, or $(ii)$ the conductivity is taken having plenty of constraints (with almost no grow in the variable $\nabla u$); however one can recover now solutions with large gradients. In the following, we will precisely specify which of these constraints are needed in our work.   

\subsection{Assumptions}

Let us come back to equation \eqref{IP}. The purpose of this paper is to give a first insight on the resolubility of the Calder\'on's inverse problem for the most possible general quasilinear problem.  However, unlike equation \eqref{Calderon}, the solvability (i.e. existence, uniqueness) of the \emph{direct problem} \eqref{IP} is not guaranteed in general.  Indeed, before stating our main results, we will need to assume some standard \emph{structural assumptions}\footnote{This terminology comes from Gilbarg-Trudinger's monograph \cite{GT}.} on the conductivity $a(\cdot,\cdot,\cdot)$ that will ensure the existence and uniqueness of a solution for the quasilinear direct problem. These are standard sufficient conditions, stated e.g. in Gilbarg and Trudinger's monograph \cite{GT}, but for the sake of completeness we give full details on their meaning in Section \ref{2}. Some of these conditions are necessary, meaning that the lack of a particular assumption leads to nonexistence or non uniqueness of the quasilinear solution. The reader may also consult \cite{HS} for similar conditions, in the case of a conductivity only depending on $x$  and $\nabla u$.

\medskip

\noindent
{\bf Structural assumptions.} Recall that we have assumed that $\Omega \subseteq\R^n$ is an open, bounded domain, of class $C^{2,\al}$, for some $0<\al<1$ fixed, and also that $n\geq 2$. Additionally, let us assume the following:

\medskip

\begin{itemize}
\item[(S1)]\label{S1}  (Smoothness and nonnegativity) $a\in C^{1,\al}(\overline{\Omega} \times  \R \times  \R^n)$, and $a(x,s,p)>0$ for all $(x,s,p) \in \overline\Omega \times \R\times \R^n.$

\medskip

\item[(S2)]\label{S2} (Ellipticity) Let $a_{ij} $ be the symmetric $n\times n$ matrix
\[
a_{ij}(x,s,p):= \frac12 ((\partial_{p_i} a)(x,s,p)\, p_j+(\partial_{p_j} a)(x,s,p)  \, p_i).
\]
Assume that 
\be\label{Ellip0_00}
\hbox{$a_{ij}$ is elliptic in $\overline{\Omega}$},
\ee
which means that, for all $(x,s,p) \in \overline\Omega \times \R\times \R^n$,
\[
0<\la(x,s,p) |\xi|^2 \leq a_{ij}(x,s,p) \xi_i\xi_j \leq \Lambda (x,s,p)|\xi|^2<+\infty,
\]
(see Definition \ref{Elliptic0} and \eqref{Elliptic1} for more details and a general definition).

\medskip

\item[(S3)] (Growth conditions) Additionally, we will assume the following growth conditions: for any $(x,s,p) \in \Omega \times \R\times \R^n$,
\be\label{Conds_12_0}
\begin{cases}
\begin{aligned}
\la (x,s,p)   \geq &~ \la_0 (|s|)>0  ,  \\
|p| |\nabla_p a(x,s,p)| +  |a(x,s,p)| \leq & ~ \mu_0(|s|), \\
(1+|p|)  |\partial_s a (x,s,p) |  +   |\nabla_x a(x,s,p)| \leq & ~ \mu_0(|s|) |p|,\\
p\cdot A(x,s,p) \geq &~ |p|^\bt -|a_1 s|^\bt -a_2^\bt,
\end{aligned}
\end{cases}
\ee
for functions $\la_0$ (resp. $\mu_0$) positive and non-increasing (resp. non-decreasing) in $|s|$, and constants $\bt>1$, $a_1,a_2>0$.  
\end{itemize}

It is stated in \cite{GT} (see more details in Section \ref{2}, and in particular, the descriptions in Theorem \ref{Regularity_1}) that, under assumptions (S1)-(S3), the quasilinear problem for $u:\Omega\to \R$ real-valued,
\be\label{Quasi_00}
\Dv (a(x,u,\nabla u)\nabla u) =0 \hbox{ in }\Omega, \quad u\big|_{\partial\Omega}=f\big|_{\partial\Omega}, \qquad f \in C^{2,\al}(\overline\Omega),
\ee
is {\bf uniquely solvable} for $u$ in $C^{2,\al}(\overline\Omega)$, with standard a-priori estimates.\footnote{The assumption $f \in C^{2,\al}(\overline\Omega)$ is standard for the theories developed in \cite{GT}, and it avoids the problem of finding a smooth extension of a boundary value condition $f$ defined only on $\partial\Omega$. For the sake of simplicity, we will assume this simplification as well as in \cite{GT}.} Moreover, the DN map
\be\label{Gammaa}
C^{2,\al}(\partial\Omega) \ni f \longmapsto  \Gamma_a[f] := a(x,u,\nabla u) \nabla   u \cdot \nu \big|_{\partial\Omega} \in C^{1,\al}(\partial\Omega),  
\ee
where $u=u_f$ is the solution to the equation \eqref{Quasi_00}, is well-defined and continuous (Corollary \ref{Regularity_2}).

\begin{rem}
Conditions (S1)-(S3), although sufficient for solvability of the quasilinear problem \eqref{IP}, are in some sense also necessary, because  some a-priori estimates on the boundary (needed for the existence part) may fail if one of these conditions is lifted, see Remark \ref{Rem2p16} for more details.
\end{rem}

\begin{rem}
The conditions imposed in (S3) may seem too strong compared with the standard theory for linear scalar elliptic problems, but they are required with the purpose of having solutions with large gradients. Less restrictive assumptions on the growth of $a(x,s,p)$ as a function of $p$ are certainly possible, but an a-priori restriction on the size of the gradient $\nabla u$ of the solution will be probably needed.  
\end{rem}

\medskip

\noindent
{\bf Inverse problem assumptions.}  In addition to the previous estimates, we will need three additional, non-structural assumptions. We call them non-structural assumptions because these are sufficient conditions for showing that the DN map in \eqref{Gammaa} is one-to-one. Although needed for the method of proof, we also believe that some of them are in a certain sense also necessary conditions, but we have no rigorous proof of this claim. In what follows, we shall assume the following hypotheses

\medskip

\begin{itemize}
\item[(H1)]\label{H1} The conductivity $a$ is homogeneous in space: it only depends on $s$ and $p$: $a=a(s,p)$, and $a(s,p) >0$ for all $(s,p)\in \R\times \R^n$.

\medskip

\item[(H2)]\label{H2} There are constants $r_0,R_0>0$, such that $a(s,p)$ has an \emph{analytic} continuation as a function of $(s,p)\in \mathcal R\times B_{r_0}(0)$,  where\footnote{$B_r(p_0):= \{ p\in \Com^n ~ : ~ |p-p_0|<r\}$. Note also that $\mathcal R$ is unbounded.}
\be\label{mR}
\mathcal R: =\R + i[-R_0,R_0] \subseteq\Com
\ee
is a (not necessarily small) band around the real line, and with $a(s,p) \in \Com$ being real-valued for $s$ and $p$ real-valued.

\medskip

\item[(H3)]\label{H3} Under assumption (H2), the real part of the complex-valued function $a(s,p)$ is positive, in the sense that there exist $0<\la<\Lambda<+\infty$ (depending continuously on $s$, $p$) such that\footnote{Note that we are not asking for $a$ being entire, but only bounded on a bounded domain.} 
\be\label{H30}
\begin{aligned}
0<\tilde\la\leq \la(s,p)  \leq & ~\re \, a(s,p) \\
 |a(s,p)| +| \nabla_p a(s,p)|   \leq &~ \Lambda,~ \hbox{ for all }  ~ (s,p) \in  \mathcal R\times B_{r_0}(0).
\end{aligned}
\ee

%
%

\end{itemize}

Some preliminary remarks about these conditions are absolutely necessary.

\begin{rem}
Hypothesis (H1) can be understood as a reduction on the number of variables to be found: we are looking for a conductivity depending on $n+1$ variables, but still improving in some sense each of the results in \cite{S,SunU,HS}, which recover $\leq n+1$ variables (and in the case of a nontrivial gradient, the problem simplifies to an almost linear situation). However, we also believe that hypothesis (H1) can be relaxed to allow ``near to a constant'' inhomogeneous conductivities, as in standard ``direct'' elliptic theory. Hypothesis (H2) is needed for the method of proof, and probably can be lifted after one can construct some equivalent real-valued test functions to the ones we will mention in this paper. Finally, hypothesis (H3) is needed to preserve the ellipticity of a suitable complex-valued quasilinear problem, and it is certainly an essential condition for us.  
\end{rem}

\medskip

\noindent
{\bf Examples.} Some examples of conductivities of this type are, for $p$ small if necessary, the conductivity generalizing the \emph{minimal surface equation}:
\[
a_1(s,p) := \frac{f(s)}{\sqrt{1+ p\cdot p} },
\]
among many others (in this last example one must rewrite the equation without fractional terms). See also \cite[pp. 260--263]{GT} for more details on the class of quasilinear equations appearing from different applied problems.

\subsection{Main results}
Now we state our main result. Let  $u=u_f$ be a solution of the equation
\be\label{IPnew}
\Dv (a(u,\nabla u)\nabla u) =0 \hbox{ in }\Omega, \quad u\big|_{\partial\Omega}=f.
\ee
\begin{thm}\label{MT}
Under the hypotheses \emph{(S1)-(S3)} and \emph{(H1)-(H3)}, the knowledge of the DN map
\be\label{Gamma}
C^{2,\al}(\partial\Omega) \ni f \longmapsto \Gamma_a[f] = a(u,\nabla u) \partial_\nu u\big|_{\partial\Omega} \in C^{1,\al}(\partial\Omega),  
\ee
fully and uniquely determines  the quasilinear coefficient $a(s,p)$ in problem \eqref{IPnew}. More precisely, if $\Gamma_{a_1} \equiv \Gamma_{a_2}$, then $a_1 \equiv a_2$ in $\mathcal R \times B_{r_0}(0)$.
\end{thm}

\begin{rem}
This result can be seen as the first example of uniqueness for the quasilinear Calder\'on's problem where the conductivity depends on both $u$ and $\nabla u$ in a nontrivial fashion.
\end{rem}

\begin{rem}
In principle, hypothesis (H1) may seem too restrictive, but as mentioned before, recovering a conductivity depending on $x$, $u$ and $\nabla u$ could be considered as a problem with too many degrees of freedom, and uniqueness may not hold for the case of large gradients. On the other hand, we also believe that the analyticity condition (H2) can be relaxed to allow less restrictive conductivities. 
\end{rem}

\subsection{Ideas of the proof}

The proof of Theorem \ref{MT} relies on the introduction of a new class of solutions for \eqref{IPnew} which have nontrivial gradient. Recall the hypothesis (H2), that ensures that  $a(s,p)$ is analytic in a particular tubular neighborhood of the real case. Assume that $0\in \partial\Omega$, otherwise we translate the domain (or change the following argument by a suitable space translation). Under this framework, we fix $s\in \Com$ and $p\in \Com^n$ such that $p\cdot p =0$, and introduce the following set of functions
\be\label{Usp}
u_{s,p}(x) := s+x\cdot p \in \Com, \quad x\in \Omega.
\ee
Note that $x\cdot p$ is the standard inner product between the (real-valued) vector $x$ and the complex-valued vector $p$. A first important property of these functions is the following: each $u_{s,p}$ solves \eqref{IPnew} in the classical sense, provided $p\cdot p=0$:
\[
\Dv (a(u_{s,p},\nabla u_{s,p})\nabla u_{s,p}) = 0 \quad \hbox{ in } ~\Omega.
\]
This last identity shows precisely a nontrivial bifurcation, along a complex analytic manifold ($p\cdot p=0$) of the standard constant solutions (in this case, $u_{s,0}$), which where mostly considered by Z. Sun and coauthors \cite{SunU,HS}.

\medskip

Solving in general the direct quasilinear problem \eqref{IPnew} for $a(\cdot)$ \emph{complex-valued} is a hard problem. Very few results are available in the literature, and they mostly consider  the linear case only. Some recent breakthroughs on the regularity problem are the works by Hofmann et al. \cite{Hofmann}, and subsequent papers. See also the work by Barton \cite{Barton} for more details on this approach. In our case, we will only consider solutions that are close enough to a particular exact solution $(u_{s,p})$, which in some sense simplifies the solvability theory.

\medskip

In Theorem \ref{Regularity_2_complex} we will show the existence and uniqueness of complex-valued solutions in the neighborhood of each $u_{s,p}$, provided $p$ is chosen small. In proving this result we will invoke the Implicit Function Theorem. A suitable complex-valued Dirichlet-to-Neumann map arises from this construction. 

\medskip

The second ingredient of the proof is the linearization technique above mentioned, applied this time to the complex-valued case. We will show in Corollary \ref{5p3} that if two complex-valued DN maps coincide, then their respective linearization are well-defined and must coincide, at least for small gradients.  

\medskip

A third ingredient of the proof is the uniqueness a particular set of complex-valued conductivities in a linearized Calder\'on problem. Unlike the standard real-valued Calder\'on problem, the proof of the former result relies on simple and elementary evaluation techniques, and no CGO solutions are needed (although each $u_{s,p}$ may be recast as a very particular CGO solution), see the proof of Theorem \ref{Uniqueness_linear} for the corresponding details. We also emphasize that our techniques do not apply for the standard real-valued Calder\'on problem, because of the absence of a particular zeroth order term that makes things work in our case.

\medskip

The last part of the argument consists in comparing the real and complex-valued DN maps. By hypothesis, we only have information about the real-valued one, and some information must be transferred from the real to the complex one. In order to show this fact we will prove that the complex-valued DN map is the unique continuation of the real-valued one, around each $u_{s,p}$, to the complex $n$-dimensional space $p\in \Com^n$, $p$ small. We will use here the fact that the manifold $p\cdot p=0$ is analytic outside the origin. As a consequence, complex-valued DN maps are equal for each boundary value data. Consequently, linearized DN maps coincide, and therefore, from the previous results (Theorem \ref{Uniqueness_linear}), both conductivities are the same everywhere.

\subsection*{Organization of this paper} In Section \ref{2}, and in order to make this paper self-contained, we review the standard solvability theory of quasilinear problems. Then, in Section \ref{3} we show the existence and uniqueness of a complex-valued, linear elliptic problem. In Section \ref{4} we extend this result to the nonlinear case, and show the existence of well-defined DN maps. Section \ref{5} is devoted to the linearization of the DN map, and the transference of uniqueness from nonlinear to the linear regime. Section \ref{6} deals with the uniqueness for the Calder\'on problem associated to a particular linear, complex-valued equation resulting from the linearization of the quasilinear problem. Finally, in Section \ref{7} we show the main result, Theorem \ref{MT}.


\subsection*{Notations}  Through this paper, we will assume the following conventions:
\begin{itemize}
\item Given $p\in \Com^n,$ $p^T$ denotes its transpose vector.
\item $C^\al(\Omega;\Com)$, $C^\al(\partial\Omega;\Com)$  denote H\"older spaces of  complex-valued functions of exponent $\al$.
\item $B_r(x_0)$ will denote the open ball centered at $x_0 \in \Com^n$ (or $\R^n$), of radius $r>0$.
\item $\nu(x) \in \mathbb S^{n-1}$ denotes the outer unit normal to a point $x\in \partial\Omega$.
\item Given an $m\times n$ matrix $A$, we will denote its norm as $\|A\|^2 := \sum_{i,j} |a_{ij}|^2$. 
\end{itemize}

\subsection*{Acknowledgments} C. M. would like to thank the Mathematics Department of the U. Washington for its kind hospitality during the elaboration of this work. He also would like to thank the Laboratoire de Math\'ematiques d'{}Orsay for his kind hospitality during past years, and where part of this work was completed. Finally, C.M. was partially funded by ERC Blowdisol (France), Fondecyt no. 1150202 Chile, Fondo Basal CMM (U. Chile), and Millennium Nucleus Center for Analysis of PDE NC130017. G. U. was partly supported by NSF and a Si-Yuan Professorship at IAS, HKUST. 

\bigskip

\section{Review on the real-valued, quasilinear direct problem}\label{2}

\medskip

\subsection{Preliminaries} In this section we recall some well-known results on quasilinear scalar equations. For the sake of completeness, we state (without proofs) all interesting results, even if they are not essentially needed. Some good references for these results are the monographs by Ladyzenskaja and Ural'ceva \cite{LU}, and also Gilbarg and Trudinger \cite{GT}. Most of the results below are stated for operators in divergence form, but they have general counterparts, see the aforementioned monographs for more details and general statements. 

\medskip

Recall that we assumed $\Omega$ being a smooth ($C^{2,\al}$ for instance, $0<\al<1$), bounded domain in dimension $n\geq 2$. The regularity of the boundary is essentially needed in one specific statement, Theorem \ref{Existence_0}.  Let $\mathcal Q =\mathcal Q_u$ be an operator in divergence form
\be\label{Qu}
\begin{aligned}
\mathcal Q_u u=\mathcal Q u := & ~\Dv A(x,u,\nabla u), \\
 u= &~ u(x)\in C^2(\Omega).
\end{aligned}
\ee
Here the vector field $A=A(x,s,p)$, defined in $\Omega\times \R \times \R^n$ with values in $\R^n$, is assumed to be at least differentiable, although in practice we will need $C^1$ regularity, see (S1) in page \pageref{S1}. We start with the following definition, which is standard.

\medskip

\begin{defn}[Ellipticity, see \cite{GT}, eqn. (10.5), p. 259]\label{Elliptic0}
We say that $\mathcal Q$ as in \eqref{Qu} is elliptic in $\Omega$ if there are constants $0<\la<\Lambda<+\infty$, depending on $(x,s,p)\in \Omega\times \R \times \R^n$, such that for all $\xi\in \R^n,$ 
\be\label{Elliptic1}
0<\la(x,s,p) |\xi|^2 \leq a_{ij}(x,s,p) \xi_i\xi_j \leq \Lambda (x,s,p)|\xi|^2,
\ee
and where
\be\label{aij}
 a_{ij}(x,s,p) := \frac12 (\partial_{x_j} A_i+\partial_{x_i} A_j)(x,s,p).
\ee
Similarly, we will say that $\mathcal Q$ is elliptic in $\overline{\Omega}$ if \eqref{Elliptic1} holds for all $(x,s,p) \in \overline{\Omega} \times \R\times \R^n$, and \emph{uniformly elliptic} if $\la$ and $\Lambda$ do not depend on $x\in \Omega$.
\end{defn}

\noindent
{\bf Examples.} Some examples of vector fields $A$ that we will see through this paper are the following:
\[
A_1(x,s,p):= a(x,s,p) p, \quad a(\cdot) \hbox{ scalar valued}.
\]
and for $(s,p)\in \R\times \Com^n$ fixed, and if $a=a(s,p)$ is differentiable in its  variables $(s,p)$,
\[
A_2(x,z,q):=  a(s + p\cdot x,p) q + p\, \{   (\nabla_p a)(s+ p\cdot x,p) \cdot q +(\partial_s a)(s + p\cdot x, p) z  \}.
\]
This second vector field $A_2$ can be recast as a sort of linearization of  the nonlinear field $A_1$, around a particular solution.

\begin{rem}[Complex-valued case]\label{rem_complex}
Let us assume now that $A=A(x,s,p)$ is complex-valued, differentiable in $x\in \Omega$ and analytic in a region of points $\mathcal M \ni (s,p)\in \Com\times \Com^n$. We will say that $\mathcal Q$ is {\bf elliptic} in $\Omega \times \mathcal M$ if there are constants $0<\la<\Lambda<+\infty$, depending on $(x,s,p)\in \R^n\times  \mathcal M$, such that for all $\xi\in \Com^n$,
\be\label{Elliptic1_complex}
0<\la(x,s,p) |\xi|^2 \leq \re \big( a_{ij}(x,s,p) \xi_i \overline{\xi_j}\big) \leq \Lambda (x,s,p)|\xi|^2,
\ee
and where, as usual,
\be\label{aij_complex}
 a_{ij}(x,s,p) = \frac12 (\partial_{x_j} A_i+\partial_{x_i} A_j)(x,s,p).
\ee
As in the real-valued case, we will say that $A$ is uniformly elliptic in $\Omega \times \mathcal M$ if both $\la$ and $\Lambda$ are functions not depending on $x\in \Omega$. 
\end{rem}

\subsection{Uniqueness} We first mention a uniqueness result. The following result is a slight modification of Theorem 10.7, p. 268 in \cite{GT}, adapted to our needs. 

\begin{thm}[Uniqueness]\label{Uniqueness}
Assume that $\mathcal Q u =\mathcal Qv =0$ for $u,v \in C^2(\overline{\Omega})$, with $\mathcal Q$ in \eqref{Qu} elliptic in $\Omega$ (see inequalities \eqref{Elliptic1}), and $A=A(x,s,p)$ continuously differentiable with respect to the $s$ and $p$ variables. If $u = v$ on $\partial\Omega$, then $u\equiv v$ in $\Omega$.  
\end{thm}


The next result explains the maximum principle for operators in the form of divergence, adapted to our setting. Note that no ellipticity assumption is needed, although a different \emph{coercivity} assumption is imposed.

\begin{thm}[See Theorem 10.9, p. 272 in \cite{GT}]
Assume that $u\in C^0(\overline{\Omega}) \cap C^1(\Omega)$ satisfies $\mathcal Q u =0$ in the weak sense in $\Omega$,\footnote{This means the equation $\mathcal Q u=0$, with $\mathcal Q$ and $u$ as in \eqref{Qu}, is tested against a $C_0^\infty(\Omega)$ function.} and suppose that for some $\bt> 1$, and $a_1,a_2>0$,
\be\label{Coercivity_0}
p\cdot A(x,s,p) \geq |p|^\bt -|a_1 s|^\bt -a_2^\bt, \quad \hbox{for all } (x,s,p) \in \Omega \times \R\times \R^n.
\ee
Then one has the estimate
\[
\sup_\Omega |u| \leq C(a_2 +a_1 \sup_{\partial\Omega}|u| )+ \sup_{\partial\Omega}|u|, \quad C=C(n,\bt,a_1,|\Omega|)>0.
\]  
\end{thm}

Before continuing, some remarks are essentially needed.

\begin{rem}
This result is useful because it gives a-priori $C^0$ estimates for any sufficiently smooth solution $u$ of $\mathcal Q u=0$ in $\Omega$, only in terms of its values on the boundary. In that sense, this is the first step for establishing a-priori estimates for solutions to the quasilinear problem $\mathcal Q u=0$.
\end{rem}

\begin{rem}\label{Rem_0}
Assume that $A(x,s,p) = a(x, s,p) p$, with $a \geq 1$ uniformly in $(x,s,p)\in \Omega\times \R\times \R^n$. Then we have
\[
p\cdot A(x,s,p) = a(x,s,p) |p|^2 \geq |p|^2,
\]
which implies that \eqref{Coercivity_0}  is satisfied with $\bt=2>1$, and $a_1=a_2=0$.  Therefore, in this particular case, one simply has the pure $C^0$ estimate
\[
\sup_\Omega |u| \leq \sup_{\partial\Omega}|u|.
\]
\end{rem}


\subsection{A priori estimates}

The next step is how to establish existence of solutions for the quasilinear problem $\mathcal Q u =0$, where $\mathcal Q$ is as in \eqref{Qu}. Recall the definition of ellipticity in $\overline{\Omega}$, see Definition \ref{Elliptic1}.  For the next result, we will assume, for $|p| \to +\infty$, the following \emph{structural} conditions on $A$: 
\be\label{Conds_11}
\begin{aligned}
\nu (|s|) (1+|p|)^\tau  \leq & ~ \la (x,s,p),  \quad \hbox{(see \eqref{Elliptic1}),} \\
\|\nabla_p A(x,s,p)\|  \leq & ~ \mu(|s|) (1+|p|)^\tau \\
(1+|p|) |\partial_s A (x,s,p) |  + \| D_x A(x,s,p) \| \leq & ~ \mu(|s|) |p|^{\tau +2},
\end{aligned}
\ee
for some $\tau >-1$, and $\nu$ (resp. $\mu$) positive and non-increasing (resp. non-decreasing) in $|s|$.

\medskip

The next result, essentially Theorem 14.1 in \cite{GT} (p. 337), establishes boundary gradient estimates for solutions to $\mathcal Q u=0$ in $\Omega$. Before announcing this, we need some standard notations and definitions. Assume that $u\in C^2(\Omega)$, and that $\mathcal Q$ is elliptic as in Definition \ref{Elliptic0}. Recall that \eqref{Qu} can be written, using Einstein's summation convention, as
\[
\begin{aligned}
\mathcal Qu =&~ \Dv_x A(x,u,\nabla u)  \\
=& ~ \partial_{x_i} A_i(x,u,\nabla u) \\
 =&  ~(\partial_{x_i} A_i)(x,u,\nabla u) +(\partial_{s} A_i)(x,u,\nabla u)~ \partial_{x_i} u +\partial_{p_j} A_i(x,u,\nabla u)~\partial_{x_ix_j}^2 u.
\end{aligned}
\]
Since $u\in C^2(\Omega)$, we have $(\partial_{x_ix_j}^2 u)_{i,j}$ symmetric, which implies that 
\[
\begin{aligned}
\mathcal Qu = &~ \frac12 (\partial_{p_j} A_i(x,u,\nabla u)+\partial_{p_i} A_j(x,u,\nabla u)) ~\partial_{x_ix_j}^2 u\\
 &  ~ + (\partial_{x_i} A_i)(x,u,\nabla u) + (\partial_{s} A_i)(x,u,\nabla u)~ \partial_{x_i} u \\
 =& ~ a_{i,j}(x,u,\nabla u)~\partial_{x_ix_j}^2 u + b(x,u,\nabla u),
\end{aligned}
\]
where 
\[
b(x,s,p) := (\partial_{x_i} A_i)(x,s,p) + (\partial_{s} A_i)(x,s,p)~ p_i,
\]
see also Definition \ref{Elliptic0}. Gilbarg and Trudinger define (see eqn. (10.3)) the principal part of $\mathcal Q$ as $\mathcal E$:
\[
\mathcal E(x,s,p) := a_{i,j}(x,s,p) p_i p_j,
\]
so that, since $\mathcal Q$ is elliptic,
\[
0<\la(x,s,p) |p|^2 \leq \mathcal E(x,s,p) \leq \Lambda (x,s,p)|p|^2.
\]
Having this in mind, we also have
\be\label{New_Cond_0}
\begin{aligned}
|p| \Lambda (x,s,p) \sim & ~ |p| \| a_{i,j}(x,s,p)\| \\
\sim & ~ |p| \| D_{p} A(x,s,p)\| \leq \mu(|s|) (1+|p|)^{\tau +1},
\end{aligned}
\ee
thanks to \eqref{Conds_11}. Applying once again \eqref{Conds_11}, we have
\[
\mu(|s|) (1+|p|)^{\tau +1} \leq \tilde \mu(|s|) \la (x,s,p)(1+|p|)  \leq  \tilde \mu(|s|) \mathcal E(x,s,p),
\]
for $p$ large. Additionally, if $p$ is large,
\be\label{New_Cond_1}
\begin{aligned}
|b(x,s,p)| \leq & ~ |(\Dv_x A)(x,s,p)| +\| (\partial_{s} A)(x,s,p)\||p|  \\
 \leq & ~  \mu(|s|) |p|^{\tau +2} \\
 \leq & ~ \tilde \mu(|s|) \mathcal E(x,s,p),
\end{aligned}
\ee
using again \eqref{Conds_11}. In conclusion,
\[
|p|\Lambda (x,s,p) + |b(x,s,p)| \leq \tilde \mu(|s|) \mathcal E(x,s,p).
\]
These are the ``structure'' conditions imposed in \cite[eqn. (14.9)]{GT}, and are clearly satisfied thanks to \eqref{Conds_11}. Consequently, we have

\begin{thm}[Boundary estimates]\label{Boundary}
Let $u\in C^2(\Omega)\cap C^1(\overline{\Omega})$ satisfy 
\be\label{Q_new}
\mathcal Q u= a_{i,j}(x,u,\nabla u) \, \partial_{x_ix_j}^2 u + b(x,u,\nabla u)=0
\ee
in $\Omega$ and $u\big|_{\partial\Omega}=f \in C^2(\overline{\Omega})$. Suppose that $\Omega$ satisfies the uniform exterior sphere condition, with uniform radius $\delta>0$.  Then, under assumptions \eqref{Conds_11}, one has
\be\label{Gradient_Boundary}
\sup_{\partial\Omega}|\nabla u| \leq C(n,M, \mu(M), N,\delta), \qquad M:=\sup_\Omega |u|, \quad N:= \|f\|_{C^2(\overline{\Omega})}.
\ee
\end{thm}

\begin{rem}
The uniform exterior sphere condition for $\Omega$ is needed in order to construct suitable barriers, which is a standard technique in elliptic theory. See \cite[Chapter 14]{GT} for more details.
\end{rem}

\medskip

The following result is Theorem 15.9 in \cite{GT}. Recall the definitions of $M$ and $N$ in \eqref{Gradient_Boundary}.

\begin{thm}
Let $u\in C^2(\Omega)\cap C^0(\overline{\Omega})$ satisfy \eqref{Q_new} in $\Omega$ bounded and assume \eqref{Conds_11} valid for $\tau>-1$. Assume additionally that $\Omega$ satisfies the exterior sphere condition and that $u=f$ on $\partial\Omega$, with $f\in C^2(\overline{\Omega})$. Then we have
\be\label{Gradient_Interior}
\sup_\Omega|Du| \leq C(n,\tau,\nu(M),\mu(M),\partial\Omega, N, P), \quad P:=\sup_\Omega |A(x,0,0)|.
\ee
\end{thm}

Now we recall some H\"older estimates for the gradient of $u$. Let us remind the H\"older seminorm:
\[
[u]_{\al,\Omega} := \sup_{x,y\,\in \Omega,~ x\neq y} \frac{|u(x)-u(y)|}{|x-y|^\al}.
\]

\begin{thm}[See Theorem 13.2 p. 323 in \cite{GT}]
Assume $u\in C^2(\overline{\Omega})$ is such that $\mathcal Q u=0$ in $\Omega$, with $\mathcal Q$ elliptic in $\overline{\Omega}$, and $A\in C^1(\overline{\Omega}\times \R\times \R^n)$. Finally, assume that $\partial\Omega\in C^2$ and $u=f$ on $\partial\Omega$, where $f\in C^2(\overline{\Omega})$. Then
\[
[\nabla u]_{\al,\Omega} \leq C\Big( n,K, \frac{\Lambda_K}{\lambda_K},\frac{\mu_K}{\la_K},\Omega, \|f\|_{C^{2}(\overline\Omega)} \Big), \quad K:= \|u\|_{C^{1}(\overline\Omega)},
\]
and $\al=\al(n, \Lambda_K/\lambda_K,\Omega)>0$ (see (13.4) in \cite{GT}).
\end{thm}

\medskip

\subsection{Existence. Leray-Schauder fixed point argument}
The following Leray-Schauder type result is the key tool to prove existence of solutions for $\mathcal Q u=0$. Note that the existence is proven in H\"older classes, however, this condition could be relaxed by allowing a less regular class of solutions (and a different notion of solution).

\begin{thm}
Let $\Omega$ be a bounded domain in $\R^n$, with $\mathcal Q$ as in \eqref{Q_new} elliptic in $\overline\Omega$, with coefficients 
\[
a_{ij}\in C^\al(\overline\Omega\times \R \times \R^n), \quad b\in C^\al (\overline\Omega \times\R\times \R^n), \quad 0<\al<1. 
\] 
Let $\partial\Omega \in C^{2,\al}$ and $f\in C^{2,\al}(\overline{\Omega}).$ If there exists a constant $\mathcal M$, independent of $u$ and $\sigma \in [0,1]$, such that for every $C^{2,\al}(\overline\Omega)$ solution of the Dirichlet problems
\[
\begin{aligned}
\mathcal Q_\sigma u := & ~ a_{ij}(x,u,\nabla u)\partial_{i,j}^2 u +\sigma b(x,u,\nabla u) =0 \quad \hbox{ in }~ \Omega,\\
u\big|_{\partial\Omega} = & ~ \sigma f , \quad \sigma \in [0,1],
\end{aligned}
\]
satisfies the $C^1$ uniform bound
\[
\sup_\Omega |u| + \sup_\Omega |\nabla u| <\mathcal M,
\]
then the Dirichlet problem $\mathcal Q u =0$, $u\big|_{\partial\Omega} =  f$ has a solution in $C^{2,\al}(\overline\Omega)$.
\end{thm}

\medskip

The following result is strictly contained in Theorem 15.11 in \cite{GT} (p. 381).

\begin{thm}[Existence]\label{Existence_0}
Let $\Omega$ be a bounded domain in $\R^n$ and suppose that $\mathcal Q$ is elliptic in $\overline{\Omega}$, with  $A \in C^{1,\ga}(\overline{\Omega}\times\R \times \R^n)$, $0<\gamma<1$, satisfying \eqref{Conds_11} and \eqref{Coercivity_0} with the additional restriction $\bt =\tau+2$. Then if $\partial\Omega\in C^{2,\ga}$, and for any $f\in C^{2,\ga}(\overline{\Omega})$ there exists a solution $u=u_f \in C^{2,\ga}(\overline{\Omega})$ of the problem $\mathcal Q u =0$ in $\Omega$, $u\big|_{\partial\Omega}=f$.
\end{thm}

In the following remarks, we essentially explain how Theorem \ref{Existence_0} is proved.

\begin{rem}
The assumption $\bt=\tau +2$ is required to reconcile condition \eqref{Coercivity_0} with the first estimate in \eqref{Conds_11}.
\end{rem}

\begin{rem}
For the proof of Theorem \ref{Existence_0}, several intermediate steps are needed, parts of a main strategy invoking the Leray-Schauder fixed point theorem in H\"older spaces. In order to apply this result, one needs to show some Ladyzenskaja-Ural'ceva a-priori interior and boundary estimates in $C^{1,\ga}$, for solutions to the problem $\mathcal Qu =0$ in $\Omega$, which are established through Chapters 10 , 13, 14 and 15 in \cite{GT}. See the comments after Theorem 15.11 in \cite{GT} for full details.
\end{rem}

\begin{rem}
The assumption $\partial\Omega\in C^{2,\ga}$ is needed to ensure the so-called \emph{exterior sphere condition} for every point in $\partial\Omega$. 
\end{rem}

\begin{rem}\label{Rem2p16}
Conditions \eqref{Conds_11} are only sufficient for obtaining existence for the problem $\mathcal Q u=0$, however, there are examples (see Chapter 14, Section 14.4 in \cite{GT}) that show that the absence of some of these assumptions leads to nonexistence results. Usually, the estimate that fails is the control of the derivative of $u$ on the boundary.  
\end{rem}


\medskip

\subsection{Applications} We will apply Theorems \ref{Uniqueness} and \ref{Existence_0} to show the existence of a unique solution for the direct problem associated to the quasilinear problem \eqref{IP}. A first result deals with the solvability problem for scalar quasilinear problems. Before we need some notations. Assume that 
\be\label{A=ap}
A(x,s,p) := a(x,s,p)p, \quad a\geq 1.
\ee
Then $a_{ij}$ in \eqref{aij} is given by
\[
a_{ij}(x,s,p):= \frac12 ((\partial_{p_i} a)(x,s,p)\, p_j+(\partial_{p_j} a)(x,s,p)  \, p_i).
\]
We will assume that 
\be\label{Ellip0}
\hbox{$a_{ij}$ is elliptic in $\overline{\Omega}$},
\ee
as in \eqref{Elliptic1}, with involved parametric constants  $\la(x,s,p)$ and $\Lambda(x,s,p)$ respectively (see \eqref{Elliptic1}).  Additionally, we will assume that
\be\label{Conds_12}
\begin{cases}
\begin{aligned}
\la (x,s,p)   \geq &~ \nu (|s|)>0  ,  \\
|p| |\nabla_p a(x,s,p)| +  |a(x,s,p)| \leq & ~ \mu(|s|), \\
(1+|p|)  |\partial_s a (x,s,p) |  +   |\nabla_x a(x,s,p)| \leq & ~ \mu(|s|) |p|,
\end{aligned}
\end{cases}
\ee
for $\nu$ (resp. $\mu$) positive and non-increasing (resp. non-decreasing) in $|s|$.  These conditions essentially say that $a$ must be bounded with derivatives of the right order. 

\begin{thm}\label{Regularity_1}
Consider the quasilinear problem posed in a bounded domain $\Omega \subseteq\R^n$ of class $C^{2,\al}$, $0<\al<1$, and $n\geq 2$, for $u:\Omega\to \R$ real-valued:
\be\label{Quasi_0}
\Dv (a(x,u,\nabla u)\nabla u) =0 \hbox{ in }\Omega, \quad u\big|_{\partial\Omega}=f.
\ee
Assume that $a\in C^{1,\al}(\overline{\Omega} \times \R \times  \R^n)$, and $a\geq 1$, and that \eqref{Ellip0} and \eqref{Conds_12} are satisfied. 
Then the direct problem \eqref{Quasi_0} is uniquely solvable for $u$ in $C^{2,\al}$, i.e., for any $f\in C^{2,\al}(\partial\Omega)$, there exists a unique solution $u=u_f \in C^{2,\al}(\overline{\Omega})$. 
\end{thm}

\begin{proof}
The existence part is essentially Theorem \ref{Existence_0} with $A$ given by \eqref{A=ap}, and $ \al = 2$, $\tau =0$. Note also that conditions \eqref{Conds_11} are satisfied by assuming the conditions in \eqref{Conds_12}. For the uniqueness part, it is enough to invoke Theorem \ref{Uniqueness}.
\end{proof}



\begin{rem}
Let us comment about the meaning of assumptions \eqref{Conds_12}. They state, among other things, that the conductivity must be {\bf bounded} uniformly in $x\in \Omega$ and $p \in \R^n$. In other words, it is not allowed to have e.g.
\[
a(x,s,p) \sim |p|^2.
\] 
This requirement can be understood as a {\bf smallness}  condition for the gradients of solutions to $\mathcal Q u=0$. We will see later that this condition appears in a different form in our main results.
\end{rem}

One of the main consequences of the previous result is the following existence result for the DN map. 

\begin{cor}\label{Regularity_2}
Consider the quasilinear problem \eqref{Quasi_0}. Under the assumptions and conclusions of Theorem \ref{Regularity_1}, the Dirichlet-to-Neumann map 
\be\label{DNa}
C^{2,\al}(\partial\Omega) \ni f \longmapsto \Gamma_a[f] := a(x,u,\nabla u) \nabla u \cdot \nu \big|_{\partial\Omega} \in C^{1,\al}(\partial\Omega),  
\ee
where $u=u_f$ is the solution of \eqref{Quasi_0}, is well-defined and bounded.
\end{cor}

It seems reasonable now to deal with the inverse problem associated to the quasilinear problem \eqref{Quasi_0}. However, we will see later in this paper that, if we want to recover the conductivity $a=a(x,s,p)$ for $p\neq 0$, it is better to consider complex-valued solutions for \eqref{Quasi_0}. However, the solvability theory for this type of solutions is, as far as we know, far from being completely understood. For this reason, we will have to make a digression from the standard theory and prove some particular existence theorems for complex-valued solutions of \eqref{Quasi_0}. The fact that there are explicit solutions in some particular cases will be essential for the uniqueness proof.  

\medskip

Before treating in detail the full quasilinear problem, it is somehow better to understand a simplified, complex-valued coefficients, linear problem.

\bigskip

\section{Solvability for a complex-valued linear problem}\label{3}

\medskip

\subsection{Preliminaries} Let $g\in C^{0,\al}(\overline{\Omega})$ and 
$h\in C^{2,\al}(\overline{\Omega})$, $0<\al<1$ denote two fixed ``source'' functions. The purpose of this Section is to develop a solvability theory for  the linear direct problem for the unknown function $v=v(x)$
\be\label{a0new}
\begin{aligned}
\Dv_x  \Big[ a(u_{s,p},p) \nabla v    + p\big\{ (\nabla_p a)(u_{s,p},p)  \cdot  \nabla v  +  (\partial_s a)(u_{s,p} ,p)  v  \big\}  \Big]  = & ~ g ~ \hbox{ in }\Omega,\\
 v\big|_{\partial\Omega} = & ~ h.
\end{aligned}
\ee
Note that $u_{s,p} =s+x\cdot p$, with $s\in \R$ and $p\in \Com^n$ is such that $p\cdot p =0$ (see \eqref{Usp}). From hypothesis (H2) in p. \pageref{H1}, when considering the expression $a(u_{s,p},p)$ we are using the fact that $a$ admits an analytic continuation to the region $\mathcal R \times B_{r_0}(0)$, for $p\in B_{r_0}(0)$ and $r_0$ small if needed. In that sense, this problem has complex-valued coefficients, but it still preserves its divergence form. Finally, note that the terms
\[
(\nabla_p a)(u_{s,p},p) , \quad (\partial_s a)(u_{s,p} ,p),
\]
denote the functions $\nabla_p a(s,p)$ and $\partial_s a(s,p)$ evaluated at the point $(u_{s,p} ,p)$, respectively (and if no confusion arises, we will drop the parentheses).

\medskip

Problem \eqref{a0new} is the key element to understand in this paper. It will be essential to show solvability for the quasilinear case studied in Section \ref{4}, and additionally, it will play an important role in the associated, quasilinear inverse problem (cf. Sections \ref{5} and \ref{6}). 

\medskip

It is not difficult to identify the main symbol of the problem above. It turns out that we can write \eqref{a0new} as
\[
\Dv_x (\widetilde A(x,s,p) \nabla v + \widetilde b(x,s,p) v) =g, \quad \hbox{in }\Omega, \quad v\big|_{\partial\Omega}=h, 
\]
and where ($I_n$ is the $n\times n$ identity matrix)
\be\label{Asp}
\begin{aligned}
\widetilde A(x,s,p) :=  &~ a(u_{s,p},p)I_n +  p\, \nabla_p a(u_{s,p},p)^T  \in \Com^{n\times n}, \\
\widetilde b(x,s,p) := & ~ \partial_s a(u_{s,p} ,p) \, p \in \Com^n.
\end{aligned}
\ee
Both coefficients have complex-valued components. The $n\times n$ matrix $\widetilde A_{ij}$ is not symmetric nor diagonal, and the term $p\, \nabla_p a(u_{s,p},p)^T$ may be a \emph{very large perturbation} (in terms of its absolute value) of the main part $ a(u_{s,p},p)I_n$, in such a form that it is very probable that $\widetilde A$ is no longer elliptic, so that the nature of the Dirichlet boundary value problem may be completely different to a standard one.

\medskip

The following result states that for $p$ small enough, the matrix $\widetilde A$ from \eqref{Asp} is uniformly elliptic.

\begin{lem}\label{Small}
Assume that $(s,p)\in \R \times \Com^n$, and $x\in \Omega$. Assuming $r_0 >0$ in \eqref{mR} smaller if necessary, the following is satisfied. For all $|p| <r_0$, the complex valued matrix $\widetilde A$ in \eqref{Asp} is elliptic in the sense of Remark \ref{Elliptic1_complex}, and the vector field $\widetilde b$ in \eqref{Asp} is also uniformly bounded.
\end{lem}

\begin{proof}
Since from \eqref{H30} we have $\re a(s,p) >\la(s,p) >0$, we only have to show that for $p$ small this inequality is preserved. We have
\[
\begin{aligned}
\re \big(\widetilde A_{ij}(x,s,p) \xi_i \overline{\xi_j}\big) = & ~ \re \big( a(u_{s,p},p)|\xi|^2+  p_i \partial_{p_j} a(u_{s,p},p) \xi_i \overline{\xi_j}\big)\\
=& ~ \re  a(u_{s,p},p)  |\xi|^2+  \re (p_i \partial_{p_j} a(u_{s,p},p) \xi_i \overline{\xi_j}\big).
\end{aligned}
\]
Recall that we have $|p|<r_0$. Now, since $\Omega$ is bounded, $u_{s,p} =s+x\cdot p$ lies inside a narrow horizontal band of the complex plane, of the form
\[
\R\times [- Cr_0,Cr_0], \quad C=C(\Omega)>0.
\]
Therefore, if $r_0$ is chosen small enough, 
\[
\R\times [- Cr_0,Cr_0 ] \subseteq \mathcal R \qquad \hbox{(see \eqref{mR}).}
\]
Consequently,  by hypothesis (H2) (see p. \pageref{H1}), $a$ and its derivatives are well-defined and bounded in the set ${\mathcal R} \times {B}_r(0)$. 
\[
\abs{\re (p_i \partial_{p_j} a(u_{s,p},p) \xi_i \overline{\xi_j}\big)} \leq  r_0 |\xi|^2 \times \sup_{(\tilde s,\tilde p)\in {\mathcal R} \times {B}_{r_0}(0)} |\nabla_p a(\tilde s,\tilde p)| \leq C r_0|\xi|^2. 
\]
Additionally, using \eqref{H30} and the continuity of $\la$ (taking $r_0$ smaller if necessary),
\[
\begin{aligned}
\re  a(u_{s,p},p)  \geq &~ \inf_{(\tilde s,\tilde p) \in  {\mathcal R} \times {B}_{r_0}(0) }\re  a(\tilde s, \tilde p) \\
 \geq &~ \inf_{(\tilde s,\tilde p) \in  {\mathcal R} \times {B}_{r_0}(0) }\la(\tilde s,\tilde p) \geq \tilde \la>0.
\end{aligned}
\]
Consequently, for $r_0$ small,
\be\label{Coercivity_complex}
\re \big(\widetilde A_{ij}(x,s,p) \xi_i \overline{\xi_j}\big) \geq \frac{9}{10}\tilde \la  |\xi|^2.
\ee
On the other hand, note that in the region $\mathcal R\times B_{r_0}(0)$,
\be\label{Upper_A}
\| \widetilde A(x,s,p)\| \leq \Lambda + C|p| \leq 2 \Lambda, 
\ee
and from \eqref{Asp},
\be\label{Upper_b}
| \widetilde b(x,s,p)| \leq C|p|  \leq Cr_0.
\ee
\end{proof}
\begin{rem}
Lemma \ref{Small} and hypothesis (H3) (see \eqref{H30}) can be weakened by asking for $\tilde \la$ depending on $s$, under suitable assumptions on $a(\cdot, \cdot)$ and it first derivatives, in such a form that  \eqref{Coercivity_complex} is satisfied with a positive lower bound $\tilde \la$ depending on $s$ also.
\end{rem}

\subsection{Existence for a linear complex-valued problem} Now we will apply Lemma \ref{Small} to show existence for the problem \eqref{a0new}.

\begin{thm}\label{Linear_complex}
Let $g\in C^{0,\al}(\overline{\Omega})$ and $h\in C^{2,\al}(\overline{\Omega})$, $0<\al<1$, be two fixed data. Under the assumptions of Lemma \ref{Small}, problem \eqref{a0new} has a unique, complex-valued solution $v= v_{g,h}$ in the class $C^{2,\al}(\overline{\Omega})$. Moreover, one has the estimate
\be\label{Est_apriori}
\|v\|_{C^{2,\al}} \leq C (\|h\|_{C^{2,\al}(\partial\Omega)} + \|g\|_{C^{0,\al}(\Omega)}), \quad C=C(\la,\Lambda,r_0).
\ee
\end{thm}

\begin{proof}
The proof of this result is based in the standard procedure to show solvability of linear elliptic PDEs. 

%
%
%
%
%
%
%
%

\end{proof}

We will see later, in Chapters \ref{5} and \ref{6}, that the linear problem \eqref{a0new} appears naturally  in the study of the quasilinear inverse problem \eqref{IPnew}. Theorem \ref{Linear_complex} will be applied in the next section in order to get the desired solvability for problem \eqref{IPnew}. 

\bigskip

\section{Solution for the quasilinear complex case}\label{4}

\medskip

\subsection{A model example} Now we make a small digression from the main subject of this paper. In this subsection we will consider the Calder\'on direct problem in $\Omega\subseteq \R^n$ bounded
\be\label{model}
\Dv_x (a(x) \nabla u ) =0 \quad \hbox{ in } \Omega, \quad u\big|_{\partial\Omega} =f,
\ee
where $a\in C^{1,\al}(\overline\Omega)$ is uniformly positive, and $f \in C^{2,\al}(\overline\Omega)$, with $\partial\Omega \in C^{2,\al}$, for some $0<\al<1$. Clearly \eqref{model} has a unique real-valued solution $u=u_f \in C^{2,\al}$. Moreover, there exists a solution in standard Sobolev spaces even if $f$ is assumed less regular than H\"older. 

\medskip

The problem now is to get some insight about the same problem when now $a(\cdot)$ is assumed to be complex-valued, namely $a:\Omega \longrightarrow \Com$. Since solvability for \eqref{model} in the real and complex-valued case is related to the Riesz Theorem (or Lax-Milgram Theorem), a sufficient condition to find a solution in a Sobolev space is the ellipticity condition (see Remark \ref{rem_complex}) 
\[
0<\la \leq \re a(x) \leq \Lambda<+\infty, \quad x\in \Omega.
\]
The case where $a(\cdot)$ depends now on $x$, $u(x)$ and $\nabla u(x)$ will have require some new (but standard, in view of Section \ref{2}) restrictions, because $\la$ may now depend on $u$ and $\nabla u$. It turns out that, for some particular reasons, it is good to have a good control of the dependence on $u$ of the lower bound $\la(u,\nabla u)$. This control will be important to obtain a desired ellipticity for our problem.

\subsection{Existence close to a given solution} We will consider $s\in \R$ and $p\in \Com^n$ fixed, with $p\cdot p =0$. Recall the linear affine function $u_{s,p}(x) $ defined in \eqref{Usp}. The following is clearly satisfied:

\begin{Cl}\label{Cl_0}
Assume that $a=a(s,p)$ is a complex-valued, homogeneous conductivity, analytic for $(s,p)\in \Com \times\Com^n$. Then $u_{s,p}$ solves \eqref{IPnew} with $f=u_{s,p}$, that is,
\be\label{Cl_1}
\Dv_x (a(u_{s,p}(x),\nabla u_{s,p}(x))\nabla u_{s,p}(x)) = 0 ~ \hbox{ in }~ \Omega.
\ee
\end{Cl}
\begin{rem}
For a conductivity $a(s,p)$ defined only in a portion of the complex space $\Com \times\Com^n$ (see e.g. \eqref{mR}), we need additional restrictions on $(s,p)$, depending on $x\in \Omega$. However, even in this case it is possible to show that Claim \ref{Cl_0} do hold for a subset of possible $(s,p)$.
\end{rem}
\begin{rem}
Note additionally that, for any $x_0\in \R^n$, the modified test function $\tilde u_{s,p}:= s + p\cdot(x-x_0)$ is also solution to the first equation in \eqref{IPnew}.  This sort of  ``degeneracy'' in the choice of $x_0$ is completely absorbed by simply assuming that $x_0 =0 \in \partial\Omega.$ 
\end{rem}

\medskip

The function $u_{s,p}$ is an example of a complex-valued solution to the quasilinear problem \eqref{Cl_1}, revealing the existence of a general family of solutions beyond the ones mentioned by the existence theorems in Section \ref{2}. We would like to find new solutions to \eqref{IPnew} around this explicit solutions. For this reason, we set
\[
u = u_{s,p} +v, \quad v \hbox{ unknown,} 
\]
and we will write \eqref{IPnew} in terms of the new variable $v$. The following result is the main objective of this Section, and it can be seen as an extension of Claim \ref{Cl_0}. Before, recall the definition of $\mathcal R$ in \eqref{mR}.

\begin{thm}\label{Regularity_2_complex}
Let $h\in C^{2,\al}(\overline{\Omega})$ be a fixed, complex-valued function.  Assume that the conductivity $a=a(s,p)$ is analytic in $\mathcal R\times B_{r_0}(0)$, and consider the quasilinear problem in $\Omega \subseteq\R^n$, $n\geq 2$, for $v=v(x)$ complex-valued:
\be\label{Quasi_complex}
\begin{aligned}
\Dv_x (a(u_{s,p} +v, p  + \nabla v)(p + \nabla v)- a(u_{s,p}, p)p) = &~ 0 ~\hbox{ in }~\Omega\\
v\big|_{\partial\Omega}= & ~ h.
\end{aligned}
\ee
Finally, assume that \eqref{Elliptic1_complex} and \eqref{Conds_12} are satisfied.  Then for any small $\|h\|_{C^{2,\al}(\partial\Omega)}$ and small $p\in \Com^n$ such that $p\cdot p=0$, the direct problem \eqref{Quasi_complex} is uniquely solvable for $v=v_{h}$ in $C^{2,\al}(\overline\Omega)$. 
\end{thm}

Unlike Section \ref{2}, we will prove Theorem \ref{Regularity_2_complex} following other steps: by the application of the Implicit Function Theorem to show existence and uniqueness.  Although the methods of proof are somehow similar to the ones in Section \ref{2}, it will be clear from the beginning that the addition of a complex-valued conductivity will lead to several problems, and the smallness condition on the gradients will be essential for this approach. 
  
\begin{proof}[Proof of Theorem \ref{Regularity_2_complex}]
We will make use of the Implicit Function Theorem below. Assume $t\in \R$ and $h\in C^{2,\al}(\overline{\Omega})$ given. Let us write
\[
v(x)=  h(x)  +w(x), ~ x\in \overline{\Omega}, \quad \hbox{ $w$ unknown}.
\]
Then \eqref{Quasi_complex} writes in terms of $w$,
\[
\begin{aligned}
\Dv_x \big[a(u_{s,p} + h +w , p  + \nabla h + \nabla w)(p +\nabla h +\nabla w)- a(u_{s,p}, p)p\big] =&~ 0 ~\hbox{ in }~\Omega,\\
w\big|_{\partial\Omega}= & ~ 0.
\end{aligned}
\]
In what follows, we denote by $C^{2,\al}_0(\overline{\Omega})$ the Banach space of complex-valued functions in $ C^{2,\al}(\overline{\Omega})$ which are zero at the boundary (but not necessarily its derivatives). Let us define the map $\mathcal F=\mathcal F_{s,p,h}$ such that
\be\label{F1}
\mathcal F : C^{2,\al}_0(\overline{\Omega}) \times C^{2,\al}_0(\overline{\Omega}) \mapsto C^{0,\al}(\overline\Omega),
\ee
and
\be\label{F2}
\mathcal F[h,w]:= \Dv_x \big(a(u_{s,p} + h +w , p  + \nabla h + \nabla w)(p +\nabla h +\nabla w)- a(u_{s,p}, p)p \big).
\ee
Clearly $\mathcal F$ is well-defined, and $\mathcal F[0,0] \equiv 0$. The fact that $\mathcal F$ is of class $C^1$ is a direct computation. 

\medskip

Now, for $ \tilde w \in  C^{2,\al}_0(\overline{\Omega})$ fixed, we compute $D_w \mathcal F[0,0](\tilde w).$ Since $\mathcal F$ is continuously differentiable, we have
\[
\begin{aligned}
D_w \mathcal F[0,0](\tilde w)  = & ~ \frac{d}{d\sigma}\mathcal F[0,\sigma \tilde w]\Big|_{\sigma =0} \\
 = & ~ \frac{d}{d\sigma} \Big[ \Dv_x( a(u_{s,p} +\sigma \tilde w , p+ \sigma \nabla \tilde w)(p+\sigma \nabla \tilde w) -a(u_{s,p},p)p)\Big]\Big|_{\sigma =0}\\
 =& ~ \Dv\Big( a(u_{s,p},p) \nabla \tilde w + p\big\{  (\nabla_p a(u_{s,p} , p) \cdot \nabla \tilde w )  +  (\partial_s a)(u_{s,p},p)  \tilde w \big\}  \Big).
\end{aligned}
\]
Thanks to Theorem \ref{Linear_complex}, if $|p|<r$ the homogeneous Dirichlet problem 
\[
\begin{aligned}
\Dv_x\big( a(u_{s,p},p) \nabla \tilde w + p\big\{  (\nabla_p a(u_{s,p} , p) \cdot \nabla \tilde w )  +  (\partial_s a)(u_{s,p},p)  \tilde w \big\}  \big) =& ~  g \in C^\al(\Omega), \\
\tilde w \big|_{\partial\Omega} =&~ 0,
\end{aligned}
\]
has a unique solution $\tilde w \in C_0^{2,\al}(\overline{\Omega})$, with uniform bounds. Therefore, applying the Implicit Function Theorem, for each $h\in C^{2,\al}(\overline{\Omega})$ small in norm it is possible to find a unique small solution $w=w[h] \in C^{2,\al}(\overline{\Omega}) $ of $\mathcal F[h, w[h]] \equiv 0$, which solves our problem.
\end{proof}

\subsection{Another proof for uniqueness} Although the following result can be stated for more general conductivities, as in Theorem \ref{Uniqueness} (but under additional assumptions), we will only consider the following simple statement for uniqueness. 

\begin{thm}[Uniqueness]\label{Uniqueness_new}
Let $r_0,R_0>0$ be given by hypothesis \emph{(H2)} (see \eqref{mR}). Assume that $v_1, v_2 \in C^2(\overline{\Omega})$, complex-valued, solve \eqref{Quasi_complex} with same boundary condition,  and satisfying
\be\label{MaMa}
|p|+ \| v_1 \|_{C^2(\overline\Omega)} +\| v_2 \|_{C^2(\overline\Omega)}  < r_0,
\ee
Then $v_1\equiv v_2$ in $\Omega$.  
\end{thm}

\begin{rem}
The smallness condition on $v_1$ and $v_2$ in \eqref{MaMa} is a sufficient condition for having a positive functional for the difference of two solutions. From the existence part, we see once again that smallness is in some necessary. 
\end{rem}

\begin{rem}
The proof of Theorem \ref{Uniqueness_new} is not based on any type of pointwise comparison principle (compare with section \ref{2} and the monograph \cite[Chapter 10]{GT}), since the former simply does not work for general, complex-valued equations. In that sense, an energy method is suitably more reasonable for taking into account the oscillatory behavior of complex-valued solutions.
\end{rem}

\begin{proof}[Proof of Theorem \ref{Uniqueness_new}]
We follow the ideas of the proof of Theorem 10.7 in \cite{GT}, with some important differences because several inequalities do not hold for complex-valued functions.\footnote{In particular, no comparison principle seems to hold in the complex-value case.} In that sense, the smallness character of the involved perturbation $v$ will be a key ingredient. 

\medskip

Defining $w:=u_1-u_2 = v_1-v_2$, and $v_t := tv_1 + (1-t)v_2$, for $t\in [0,1]$. Finally, for $s\in \R$ and $p\in \Com^n$ fixed, denote
\[
A(x,z,q):= a(u_{s,p} +z, p  + q)(p + q) -a(u_{s,p},p)p, \quad z\in \Com, ~q\in \Com^n.
\]
Clearly $A$ is analytic in the $(z,q)$ variables. Then we have
\be\label{Equation_00}
\begin{aligned}
0 =& ~ \Dv (A(x,u_{s,p} + v_1, p + \nabla v_1) -A(x,u_{s,p} + v_2, p + \nabla v_2)) \\
 = &~ \partial_{i} (a_{ij}(x)\partial_j w + b_i(x) w) \qquad \hbox{(Einstein's summation convention),}
\end{aligned}
\ee
%
where, for $t\in [0,1]$,
\[
\begin{aligned}
a_{ij}(x,s,p) := & ~ \frac12 \int_0^1( \partial_{p_j}A_i +\partial_{p_i}A_j)  (x,  v_t, \nabla v_t)dt \\
b_i(x,s,p) := & ~  \int_0^1 (\partial_{s}A_i) (x,v_t, \nabla v_t)dt.
\end{aligned}
\]
A further simplification reveals that\footnote{$\delta_{ij}$ represents the standard Kronecker's identity matrix.}
\[
\begin{aligned}
a_{ij}(x,s,p) =  &~ \delta_{ij} \int_0^1a (u_{s,p}+ v_t, p+ \nabla v_t)dt  \\
& ~ + \frac12 \int_0^1( p_j \partial_{p_i}a + p_i \partial_{p_j}a) (u_{s,p}+ v_t, p+ \nabla v_t)dt,
\end{aligned}
\]
and
\be\label{BB}
b_i(x,s,p) = p_i \int_0^1 (\partial_{s}a) (u_{s,p} + v_t, p+\nabla v_t)  dt
\ee
Using (H3) (see p. \pageref{H3}) and the smallness hypotheses on the perturbation \eqref{MaMa}, we note that $a_{ij}$ is still elliptic, in the sense of Remark \ref{Elliptic1_complex}:
\[
0<\la |\xi|^2 \leq \re \big( a_{ij}(x,s,p) \xi_i \overline{\xi_j}\big) \leq \Lambda |\xi|^2,
\]
and
\[
|a_{ij}| \leq \Lambda, \qquad |b_i| \leq C.
\]
Therefore, using \eqref{Equation_00} with test function $\varphi =w$,
\[
0= \re\int (a_{ij}(x)\partial_j w + b_i(x) w) \overline{ \partial_{i}w} =\re\int a_{ij}(x)\partial_j w \overline{ \partial_{i}w} + \re \int b_i(x) w\overline{\partial_i w} .
\]
The first integral above is nonnegative (ellipticity), and the second one is small, because of the smallness assumptions: from \eqref{BB} and the fact that $|p|<r_0$,
\[
\sup_{x\in \overline{\Omega}}| \re b(x)| \leq Cr_0.
\]
Therefore, using \eqref{Elliptic1_complex} and the Poincar\'e's inequality (since $\Omega$ is bounded), we have for $r$ small
\[
0 \geq \la\int |\nabla w|^2 - Cr_0^2 \int |w|^2 \gtrsim \int |\nabla w|^2.
\]
Therefore, $w\equiv 0$.
\end{proof}

As in Section \ref{2}, Corollary \ref{Regularity_2}, one of the main consequences of Theorem \ref{Regularity_2_complex} is the existence of a suitable complex-valued DN map around the solution $u_{s,p}$. 

\begin{cor}\label{Regularity_3}
Under the assumptions (and conclusions) of Theorem \ref{Regularity_2_complex}, the following holds. For any $s\in \R$, $\|h\|_{C^{2,\al}(\partial\Omega)}$ small, and $p\in \Com^n$ also small, such that $p\cdot p =0$ is satisfied, the Dirichlet-to-Neumann map 
\be\label{DN_complex}
\begin{aligned}
 C^{2,\al}(\partial\Omega) ~& \longrightarrow &  C^{1,\al}(\partial\Omega;\Com)  \\
 h \hbox{ \quad } ~ & \longmapsto  &  \widetilde\Gamma_{a,s,p}[h] \quad 
\end{aligned}
\ee
where 
\be\label{DN_complex_1}
 \widetilde\Gamma_{a,s,p}[h] := ~ \Big(a(u_{s,p} + v,p +\nabla v) (p+ \nabla v) - a(u_{s,p}, p)p\Big) \cdot \nu \big|_{\partial\Omega}
\ee
and $v=v_{h}$ is the solution of \eqref{Quasi_complex}, is well-defined and bounded.
\end{cor}

Now that we have a well-defined DN map for the quasilinear problem \eqref{Quasi_complex}, it is time for asking the differentiability properties of $\widetilde\Gamma_{a,s,p}$. 

\bigskip

\section{Linearization of the DN map}\label{5}

\medskip

\subsection{Differentiability of the DN map} In what follows, we prove that the DN map $\widetilde\Gamma_{a,s,p}$ defined in \eqref{DN_complex}-\eqref{DN_complex_1} has a well-defined Gateaux-derivative at the origin. 

\medskip

Given $h\in C^{2,\al}(\overline\Omega)$ fixed, $s \in \R$, $p\in \Com^n$ and $t\in \R$ small enough (such that Theorem \ref{Regularity_2_complex} holds for the boundary data $th$), and consider the DN operator $\widetilde\Gamma_{a,s,p} (th) $. The following result shows that this nonlinear operator has a well-defined G\^ateaux-derivative around $u_{s,p}$.

\begin{lem}\label{5p1}
Given $s,p\in \Com\times \Com^n$ and $t\in \R$ fixed, and $h\in C^{2,\al}(\overline\Omega)$. As a functional from $t\in \R$ into $C^{1,\al}(\partial\Omega)$, we have that
\[
t\mapsto \widetilde\Gamma_{a,s,p} [th]
\]
is G\^ateaux differentiable at $t=0$. Moreover, one has (cf. \ref{Asp})
\be\label{Derivative}
\lim_{t\to 0}\norm{ \frac1t(\widetilde\Gamma_{a,s,p} [th]-\widetilde\Gamma_{a,s,p} [0]) - \big( \widetilde A(\cdot ,s,p) \nabla h + \widetilde b(\cdot ,s,p)h \big) \cdot \nu }_{C^{1,\al}(\partial\Omega)} =0.
\ee
\end{lem}

\begin{rem}
Identity \eqref{Derivative} can be recast as a directional derivative:
\[
D\widetilde\Gamma_{a,s,p} [0] (h)= \big( \widetilde A(\cdot ,s,p) \nabla h + \widetilde b(\cdot ,s,p)h \big) \cdot \nu.
\]
for any $h\in C^{2,\al}(\overline\Omega)$ and $\widetilde A$, $\widetilde b$ defined in \eqref{Asp}.
\end{rem}
\begin{proof}[Proof of Lemma \ref{5p1}]
From \eqref{DN_complex_1} and \eqref{Quasi_complex}, we have
\[
\begin{aligned}
\widetilde\Gamma_{a,s,p} [t h] -\widetilde\Gamma_{a,s,p} [0] & = \big(a(u_{s,p} + v, p +\nabla v)(p+\nabla v)   - a(u_{s,p},p) p \big) \cdot \nu\big|_{\partial\Omega} \\
& =\big(a(u_{s,p} + th, p +t\nabla h)(p+t\nabla h)   - a(u_{s,p},p) p \big) \cdot \nu.
\end{aligned}
\]
The term above can be expanded as follows:
\[
\begin{aligned}
& a(u_{s,p}+ th, p +t\nabla h)(p+t\nabla h)   - a(u_{s,p},p) p  =\\
& \quad = \big( a(u_{s,p} ,p) + t\partial_s a(u_{s,p},p) h + t\nabla_p a(u_{s,p},p) \cdot \nabla h \big)(p+t\nabla h)  - a(u_{s,p},p) p \\
&\quad \quad  + t^2 D_{s,p}^2 a(\tilde s, \tilde p) [h,\nabla h]^2 (p+t\nabla h)\\
& \quad =  t \big( a(u_{s,p} ,p) \nabla h + p \partial_s a(u_{s,p},p) h + p \nabla_p a(u_{s,p},p) \cdot \nabla h \big) \\
& \quad \quad + t^2 \big( \partial_s a(u_{s,p},p) h + \nabla_p a(u_{s,p},p) \cdot \nabla h \big)\nabla h \\
&\quad \quad  + t^2 D_{s,p}^2 a(\tilde s, \tilde p) [h,\nabla h]^2 (p+t\nabla h) .
\end{aligned}
\]
Since $h\in C^{2,\al}(\overline\Omega)$, and $a$ is analytic in the considered region, we have then
\[
\begin{aligned}
\lim_{C^{1,\al}, t\to 0} \frac1t [\widetilde\Gamma_{a,s,p} [ t h] -\widetilde\Gamma_{a,s,p} [0 ] ] & =  a(u_{s,p} ,p) \nabla h \cdot \nu \\
 & \quad +  \{   \nabla_p a(u_{s,p},p) \cdot \nabla h +\partial_s a(u_{s,p}p) h  \}  p \cdot \nu.
\end{aligned}
\]
Therefore, $t\mapsto \widetilde\Gamma_{a,s,p} [t h]$ is differentiable at $t=0$ and its derivative is given by \eqref{Derivative}.
\end{proof}

\subsection{The linearized Calder\'on's problem} In what follows, we consider the following linear direct problem: given $h\in C^{2,\al}(\overline\Omega)$, find $v \in C^{2,\al}(\overline{\Omega})$ such that
\be\label{a0_new}
\Dv_x  \Big[ a(u_{s,p},p) \nabla v    + p\big\{ \nabla_p a(u_{s,p},p)  \cdot  \nabla v  +  \partial_s a(u_{s,p} ,p)  v  \big\}  \Big]  =0 \quad \hbox{ in }\Omega,
\ee
under the Dirichlet boundary condition
\be\label{a1_new}
v\big|_{\partial\Omega}= h.
\ee
Recall that in this problem, the coeffcients
\[
 a(u_{s,p},p), \quad p\, \nabla_p a(u_{s,p},p), \quad \hbox{ and } \quad p \,  \partial_s a(u_{s,p} ,p),
\]
are functions of $(x,s,p)$, but independent of $v$. Thanks to Theorem \ref{Linear_complex}, this problem has a unique solution $v=v_h \in C^{2,\al}(\overline{\Omega})$, provided $p$ is chosen small enough, a condition that will be assumed from now on. 

\medskip

For problem \eqref{a0_new}-\eqref{a1_new}, we denote by $\widetilde\Gamma_{\ell,a,s,p}[h]$ the associated DN map:
\be\label{Gamma_ele} 
\begin{aligned}
\widetilde\Gamma_{\ell,a,s,p}[h]  := & ~ \Big[ a(u_{s,p},p)  \nabla v  + p\big\{ \nabla_p a(u_{s,p},p)  \cdot  \nabla v  +  \partial_s a(u_{s,p} ,p)  v  \big\} \Big] \cdot \nu\Big|_{\partial\Omega}  \\
=& ~  \big( \widetilde A(x,s,p) \nabla h + \widetilde b(x,s,p)h \big) \cdot \nu. \qquad \hbox{ (see \eqref{Asp}).}
\end{aligned}
\ee 
Recall the nonlinear DN map $\widetilde\Gamma_a[h]$ in \eqref{DN_complex}-\eqref{DN_complex_1}. Our next result is the following

\begin{cor}\label{Coincide_0}
Fix a complex-valued coefficient $a=a(s,p)$ satisfying \emph{(H1)-(H3)} in p. \pageref{H1}, and $h \in C^{2,\al}(\overline\Omega)$. Then, for each function $(s,p) \in \R\times \Com^n$ such that $p\cdot p=0$ and $p$ is small enough, the associated DN map $\widetilde\Gamma_a$ satisfies
\[
\frac{d}{dt} \widetilde\Gamma_{a,s,p} [ t h]\Big|_{t=0} =\widetilde\Gamma_{\ell,a,s,p}[h]. 
\]
\end{cor}

\begin{proof}
Direct from \eqref{Derivative} and \eqref{Gamma_ele}.
\end{proof}

A second result translates the information from $\widetilde\Gamma_{a,s,p}$ into $\widetilde\Gamma_{\ell,a,s,p}$.

\begin{cor}\label{5p3}
Assume that $\widetilde\Gamma_{a_1,s,p} \equiv \widetilde\Gamma_{a_2,s,p}$ for given coefficients $a_1,a_2$ satisfying \emph{(H1)-(H3)} in p. \pageref{H1}. Fix $(s,p)\in \R\times\Com^n$, with $p\cdot p =0$ and $p$ small.  Finally let $\widetilde\Gamma_{\ell,a_1,s,p}$ and $\widetilde\Gamma_{\ell,a_2,s,p}$ be the corresponding DN maps obtained form the previous result for $a_1$ and $a_2$, respectively. Then, for each fixed $(s,p)$ one has 
\[
\widetilde\Gamma_{\ell,a_1,s,p} \equiv \widetilde\Gamma_{\ell,a_2,s,p}.
\]
\end{cor}

\begin{proof}
We follow the proof in \cite{SunU}. Fix $h\in C^{2,\al}(\overline\Omega)$, and pick $t\in \R$ small. We know by hypothesis that,  for any suitable $(s,p)$,
\[
\widetilde\Gamma_{a_1,s,p}[th] = \widetilde\Gamma_{a_2,s,p}[ th].
\]
In particular, thanks to Corollary \ref{Coincide_0}, both Gateaux-derivatives coincide:
\[
\frac{d}{dt}\widetilde\Gamma_{a_1,s,p}[ th]  \Big|_{t=0}= \frac{d}{dt}\widetilde \Gamma_{a_2,s,p}[ th]\Big|_{t=0},
\]
namely
\[
\widetilde\Gamma_{\ell,a_1,s,p}[h] =\widetilde\Gamma_{\ell,a_2,s,p}[h],
\]
as desired.
\end{proof}

\bigskip

\section{Uniqueness of a linear Calder\'on's problem}\label{6}

\medskip

\subsection{Setting} Let us assume, as already required in this paper, that $s\in \R$ is fixed and $p\in \Com^n$ is a small, complex-valued vector satisfying the compatibility condition $p\cdot p =0$. The purpose of this Section is the study of the Calder\'on's inverse problem for the complex-valued, matrix-valued, linear equation \eqref{a0_new}-\eqref{a1_new}-\eqref{Gamma_ele}. We also assume $h\in C^{2,\al}(\overline{\Omega})$ in \eqref{a1_new}.

\medskip

Consider the associated DN map $\widetilde\Gamma_{\ell,a,s,p}[h] $, $h$ as above, see \eqref{Gamma_ele}. Now the question is the following: can the knowledge of $\widetilde\Gamma_{\ell,a,s,p}$ for a sufficiently large class of boundary data $h$ determine the \emph{linearized}, complex-valued conductivity $a=a(z,q)$?

\medskip

Let us start by recalling the following well-known result, valid for the scalar, real-valued conductivity case, but whose extension to the complex-valued case is immediate from the proposed proof.

\begin{thm}[Sylvester and Uhlmann, \cite{SU}]\label{S_U}
Assume that $n\geq 3$ and $a=a(x)$ is an unknown $C^2(\overline{\Omega})$, real-valued conductivity. Consider the associated  Calder\'on's inverse problem for $a(x):$ 
\be\label{Calderon_0}
\begin{aligned}
\Dv (a(x) \nabla v) & = 0 \quad \hbox{in } \Omega, \\
v& = h \quad \hbox{on } \partial\Omega,
\end{aligned} 
\ee
and let us assume that two conductivities $a_1,a_2$ above produce the same DN map:
\[
H^{1/2}(\partial\Omega) \ni h \xmapsto{ \quad }  a_j(x) \nabla u \cdot \nu \big|_{\partial\Omega} \in H^{-1/2}(\partial\Omega).
\]
Then $a_1 \equiv a_2.$
\end{thm}

\begin{rem}
The proof of this result is mainly based on the use of complex geometric optics solutions (CGO). Later, we will see that our uniqueness proof will not use CGO solutions, since a very helpful, zero-order linear term will appear in the linearization of the DN map. Such a term is not present in Theorem \ref{S_U}, which makes our proof not suitable for \eqref{Calderon_0}.
\end{rem}

The purpose of this section is to extend Theorem \ref{S_U} to the complex-valued case given in \eqref{a0_new}-\eqref{a1_new}. Since this new problem is no longer a scalar one, we will need some different techniques. Before proving this result, we need an auxiliary lemma. 

\begin{lem}\label{Aux_lemma}
Assume that $0\in \partial\Omega$, $\nu(0)$ being the outer unit normal to $0\in \partial\Omega$, and let $\mathcal A(r_0)$ and $\mathcal C(r_0)$ be the sets of the form
\be\label{Ar_0}
\mathcal A(r_0) := \big\{ p\in \Com^n ~ : ~  |p|<r_0, ~p\cdot p=0 \big\},
\ee
and
\be\label{Br_0}
\mathcal C(r_0) := \big\{ p\in \Com^n ~ : ~  |p|<r_0, \, ~p\cdot p=0 ~ \hbox{ and } ~ p\cdot \nu(0) \neq 0\big\},
\ee
where $\nu(0)$ is the outer unit normal to $\Omega$ at the point $x=0$. Then, 
\begin{enumerate}
\item\label{(1)} For any $n\geq 2$, $\mathcal A(r_0)$ is not open, but $\mathcal A(r_0)\backslash \{0\}$ is an analytic manifold.
\smallskip
\item\label{(2)} For any $n\geq 2$, the vector $0\in \Com^n$ satisfies $0\not\in \mathcal C(r_0)$.
\smallskip
\item\label{(3)} For any $n\geq 2$, $ \mathcal C(r_0) \subset B_{r_0}(0) \backslash \{0\}$ in $\Com^n$.
\smallskip
\item\label{(4)} For any $n\geq 2$, $\mathcal C(r_0)$ is a complex analytic manifold.
\smallskip
\item\label{(5)} If $p \in \mathcal C(r_0)$ and $\la\in \Com \backslash \{0\}$ is such that $|\la|<1$, then $\la p \in \mathcal C(r_0) $. 
\smallskip
\item\label{(6)} If $n= 2$, $\mathcal C(r_0) $ is of the form
\be\label{Caso_2}
p = p_0 + i p_0^\perp, \quad p_0 \in \R^2 \backslash \{0\}, \quad p_0^\perp \cdot p_0=0, \quad |p_0|<\frac{r_0}{\sqrt{2}}. 
\ee
 \item\label{(7)} If $n\geq 3$, the plane passing by the origin and determined by the directions $p_r,p_i \in \R^n$, with $p=p_r +i p_i \in \mathcal C(r_0) $, is not the plane orthogonal to $\nu(0)$. 
\end{enumerate}
\end{lem}

\begin{proof}
The first assertion is a consequence of the fact that on $\mathcal A(r_0)\backslash \{0\}$,  $\nabla (p\cdot p) = p \neq 0$. The second and third assertions are direct.  The proof of \eqref{(4)} is similar to the proof of \eqref{(1)}, and the fact that $p\cdot \nu(0) \neq 0$ is an open set. Let us show \eqref{(5)}. Assume $p\in \Com^n$ small. Writing $p =p_r + i p_i$, with $p_r,p_i \in \R^n$, we have $p\cdot p=0$ if and only if
\be\label{pp0}
|p_r|^2 =|p_i|^2, \quad p_r\cdot p_i =0.
\ee
Additionally, the condition $p\cdot \nu(0) \neq 0$ reads
\be\label{pp1}
 p_r \cdot \nu(0) \neq 0 \quad \hbox{ or } \quad  p_i \cdot \nu(0) \neq 0.
\ee
In two dimensions, condition \eqref{pp1} is always satisfied if $p\neq 0$, therefore, the set of complex-valued vectors $p\in \Com^n$ for which \eqref{pp0} and \eqref{pp1} are satisfied is of the form \eqref{Caso_2}.

\medskip

In dimensions $n\geq 3$, the set of points $p_r,p_i \in \R^n$ for which $p_r \cdot \nu(0) =0$ and $p_i \cdot \nu(0) =0$ lie on a plane in $\R^n$ passing through zero. Therefore, any pair of orthogonal, equal-size vectors $p_r,p_i\in \R^n\backslash \{0\}$ for which one of them is not in the plane orthogonal to $\nu(0)$, form a satisfactory $p =p_r+ip_i$. This shows \eqref{(7)}. 
\end{proof}

Recall the region $\mathcal R$ defined in \eqref{mR}.

\begin{thm}\label{Uniqueness_linear_0}
Assume that $n\geq 2$, $s\in \R$ and $p\in \Com^n$, $p\cdot p=0$ and $|p|<r_0$ small enough such that Theorem \ref{Linear_complex} is valid. Consider two conductivities $a_1(s,p)$ and $a_2(s,p)$, defined in $\mathcal R \times B_{r_0}(0)$, and satisfying hypotheses \emph{(H1)-(H3)} in p. \pageref{H1}. Consider the Calder\'on's inverse problem associated to the linear equation \eqref{a0_new}-\eqref{a1_new}-\eqref{Gamma_ele}, and let us assume that two conductivities $a_1,a_2$ produce the same linearized DN map:
\[
 \widetilde\Gamma_{\ell,a_1,s,p}[h] = \widetilde\Gamma_{\ell,a_2,s,p}[h]   \in C^{1,\al}(\partial\Omega), \quad \hbox{ for all } h \in C^{2,\al}(\overline{\Omega}).
\]
Then $a_1 \equiv a_2$ in the region $\mathcal R \times \mathcal C(r_0)$.
\end{thm}

\begin{proof}[\bf Proof of Theorem \ref{Uniqueness_linear_0}]
The proof is simple and does not require a deep understanding or improvement of Theorem \ref{S_U}. Indeed, assume that  given $h \in C^{2,\al}(\overline{\Omega})$, we have knowledge of the DN map 
\be\label{Gamma_constant}
 \widetilde\Gamma_{\ell,a,s,p}[h] \in C^{1,\al}(\partial\Omega).
\ee
note that, since $p\cdot p =0$ ($p\in \mathcal A(r_0)$), every constant $v=c\in \Com$ is solution to \eqref{a0_new}-\eqref{a1_new} with $h\equiv c$ small. Since the solution to this problem is unique for $|p|$ small (Theorem \ref{Linear_complex}), we have that necessarily (see \eqref{Gamma_ele})
\[
 \widetilde\Gamma_{\ell,a,s,p}[c](x) =   c \, \partial_s a(u_{s,p} ,p)  (p \cdot \nu )(x).
\]
If $c\neq 0$ and $(p \cdot \nu )(x) \neq 0$ for a fixed $x\in \partial\Omega$, we will have 
\[
(\partial_s a)(s+x\cdot p  ,p) = \frac{ \widetilde\Gamma_{\ell,a,s,p}[c](x)}{c (p\cdot \nu)(x)}.
\]
Since $0\in \partial\Omega$ (otherwise we translate the domain or define $u_{s,p} = s+p\cdot(x-x_0)$, with $x_0\in \partial\Omega$ fixed), we have
\[
(\partial_s a)(s,p) = \frac{ \widetilde\Gamma_{\ell,a,s,p}[c](0)}{c (p\cdot \nu)(0)},
\]
for all $p \in \mathcal C(r_0)$ (see \eqref{Br_0}), and any $s\in \R$. From now we fix a constant $c=c_0 \in \Com\backslash\{0\}$, small if necessary. Note that the denominator is independent of $s\in \R$. Certainly we have lot of information about $a(s,p)$, because we have proved
\begin{Cl}
For any $s\in \R$ and $p\in \mathcal C(r_0)$,
\be\label{Casi}
a(s,p) = \int_0^s \frac{ \widetilde\Gamma_{\ell,a,s,p}[c_0](0) }{c_0 (p\cdot \nu)(0)}  d\tilde s + a(0,p)   \qquad (=: \widetilde a(s,p) + a(0,p)).
\ee
Moreover, $ \widetilde a(s,p)$ is completely known from the DN map, being the same for $a_1$ and $a_2$. 
\end{Cl}
Hence $a_1(s,p) =a_2(s,p)$ for any $s\in \R$ and  $p \in \mathcal C(r_0)$ small, except for a function depending on $p$ only. In order to show that the additional function $a(0,p)$ is identically the same for $a_1$ and $a_2$, we consider now the function
\[
v_p(x) := p\cdot x, \quad p \in \mathcal C(r_0).
\]
Note that $v_p$ is clearly a solution  to \eqref{a0_new}-\eqref{a1_new} with $h= p\cdot x.$ Indeed,
\[
\begin{aligned}
 \Dv_x  \Big[ a(u_{s,p},p) \nabla v_p &   + p\big\{ \nabla_p a(u_{s,p},p)  \cdot  \nabla v_p  +  \partial_s a(u_{s,p} ,p)  v_p  \big\}  \Big]=  \\
 =& ~ \Dv_x  \Big[ a(u_{s,p},p) p    + p\big\{ \nabla_p a(u_{s,p},p)  \cdot  p  +  \partial_s a(u_{s,p} ,p)  (p\cdot x)  \big\}  \Big]  \\
 =&~  (\partial_s a)(u_{s,p},p) (p\cdot p)    +  \partial_{x_j} (\partial_{p_i} a(u_{s,p},p))   p_i p_j  \\
&  ~ \quad +  (p\cdot p) (\partial_s a)(u_{s,p} ,p)  + (p\cdot p) (\partial_s^2 a)(u_{s,p} ,p)  (p\cdot x)  \\
 =& ~  \partial_{s,p_i}^2 a(u_{s,p},p))  p_i (p\cdot  p) = 0.
\end{aligned}
\]
 Since the solution to this problem is unique for $|p|$ small (Theorem \ref{Linear_complex}), we have that necessarily (see \eqref{Gamma_ele})
\[
 \widetilde\Gamma_{\ell,a,s,p}[p\cdot x] (x) =   \Big[ a(u_{s,p},p) p  + p\big\{ \nabla_p a(u_{s,p},p)  \cdot p +  \partial_s a(u_{s,p} ,p)  (p\cdot x)  \big\} \Big] \cdot \nu\Big|_{\partial\Omega} .
\]
Since $a(s,p)$ is almost completely explicit, except for a function of $p$, we have, for $x\in \partial\Omega$,
\[
\begin{aligned}
\widetilde\Gamma_{\ell,a,s,p}[p\cdot x] (x) = &  ~  \widetilde a(u_{s,p},p) (p\cdot \nu)  + a(0,p) (p\cdot \nu)  \\
& ~ + (p\cdot \nu) \big\{ \nabla_p \widetilde a(u_{s,p},p)  \cdot p +\nabla_p a(0,p)\cdot p  +  \partial_s \widetilde a(u_{s,p} ,p)  (p\cdot x)  \big\} .
\end{aligned}
\]
Evaluating at $x=0$, we have
\[
\widetilde\Gamma_{\ell,a,s,p}[p\cdot x] (0) =   ~  (p\cdot \nu(0)) \Big[ \widetilde a(s,p) +\nabla_p \widetilde a(s,p)  \cdot p + a(0,p)   +\nabla_p a(0,p)\cdot p\big],
\]
or
\[
a(0,p)   +\nabla_p a(0,p)\cdot p =  \frac1{ (p\cdot \nu(0)) }\big[\widetilde\Gamma_{\ell,a,s,p}[p\cdot x] (0) - (p\cdot \nu(0)) \big\{\widetilde a(s,p) +\nabla_p \widetilde a(s,p)  \cdot p\big\}\big].
\]
The right side above is known, and we only need to find $a(0,p)$. Now, for any $\eta>0$ we have
\[
\begin{aligned}
\frac{d}{d\eta} (\eta \, a(0,\eta p) ) = & ~ a(0, \eta p)   + \nabla_p a(0,\eta p)\cdot \eta p  \\
 =&~  \frac1{ \eta p\cdot \nu(0) }\big[ \widetilde\Gamma_{\ell,a,s,\eta p}[\eta p\cdot x] (0)  \\
 & \qquad \qquad \qquad  - (\eta p\cdot \nu(0)) \big\{\widetilde a(s, \eta p) +\nabla_p \widetilde a(s,\eta p)  \cdot \eta p\big\}\big],
\end{aligned}
\]
so that
\[
a(0, p) = \lim_{\theta\to 0}\int_\theta^1 \Bigg(  \frac{ \widetilde\Gamma_{\ell,a,s,\eta p}[\eta p\cdot x] (0)  - (\eta p\cdot \nu(0)) \big\{\widetilde a(s, \eta p) +\nabla_p \widetilde a(s,\eta p)  \cdot \eta p\big\} }{ \eta p\cdot \nu(0) }\Bigg)d\eta. 
\]
Note that the first term in the integral above must converge near $\eta=0$ since $a$ is by hypothesis analytic near the origin. Hence we have $a_1\equiv a_2$ in a set of the form $(s,p)\in \R \times \mathcal C(r_0)$, which is a complex (noncompact) analytic manifold. Since both $a_1$ and $a_2$ are analytic as functions of $s \in \mathcal R$ only, we conclude that both coincide in the region $\mathcal R \times \mathcal  C(r_0)$. The proof is complete.
\end{proof}

\medskip

We need to extend the equality between $a_1$ and $a_2$ from the set $\mathcal R \times \mathcal  C(r_0)$ to a larger set. This is a sort of unique continuation property.

\medskip

\begin{thm}\label{Uniqueness_linear}
One has $a_1\equiv a_2$ in $\mathcal R \times B_{r_0}(0)$.
\end{thm}

\begin{proof}
In what follows, we fix $s\in \mathcal R$ and $p \in \mathcal C(r_0)$. Note that from Theorem \ref{Uniqueness_linear_0}, $a_1(s,p)=a_2(s,p)$. Since $\{0\} \times \mathcal C(r_0) \subseteq \mathcal R \times \mathcal C(r_0)$ is an analytic manifold of codimension two in $\Com^{n+1}$, $n+1\geq 3$, by Riemann's second extension Theorem \cite[Theorem 2, p. 30]{Gunning}, we get the desired result.
%
\end{proof}


%
%
%

\bigskip

\section{Uniqueness for the nonlinear problem}\label{7}

\medskip

\subsection{Preliminaries} In this Section we finally prove Theorem \ref{MT}. The main idea of the proof is to find the correct link between the DN maps $\Gamma_a$ and $\widetilde  \Gamma_{a,s,p}$ already defined in \eqref{DNa} and \eqref{DN_complex}-\eqref{DN_complex_1}. We start with a simple result.

\begin{lem}\label{step_0}
Let $\Gamma_a$ the real-valued DN map from Corollary \ref{Regularity_2}, and $\widetilde \Gamma_{a,s,p}$ the complex-valued DN introduced in \eqref{DN_complex}-\eqref{DN_complex_1}. Then, for any $s\in \R$, and for any small, real-valued $h\in C^{2,\al}(\overline\Omega)$,
\be\label{paso1}
\Gamma_a[u_{s,0}+h] = \widetilde \Gamma_{a,s,0}[h].
\ee
\end{lem}

\begin{proof}
Since $u_{s,0} =s+h$ is real-valued, and since $p\cdot p=0$ if $p=0$, Theorems \ref{Regularity_1} and \ref{Regularity_2_complex} apply, with $\widetilde \Gamma_{a,s,0}[h]$ real-valued. From the uniqueness of the solutions in those theorems, we conclude \eqref{paso1}.
\end{proof}

The next definition says that it is possible to extend $\Gamma_a$ in a particular, complex-valued case. Recall the definition of $\mathcal A(r_0)$ in \eqref{Ar_0}.

\begin{defn}[Extension of $\Gamma_a$]\label{Extension}
Fix $h\in C^{2,\al}(\overline\Omega)$ with sufficiently small norm.  Then, for any $s\in \mathcal R$ and $p\in \mathcal A(r_0)$ we define the function
\[
(s,p) \xmapsto{\quad} \Gamma_a[u_{s,p} +h] \in C^{1,\al}(\partial\Omega;\Com)
\]
as follows:
\be\label{Extension_def}
\Gamma_a[u_{s,p} +h] := \widetilde \Gamma_{a,s,p}[h] \qquad (\hbox{cf.} ~\eqref{DN_complex_1}).
\ee
\end{defn}

\begin{rem}
Note that, in virtue of Lemma \ref{step_0}, the above definition of $\Gamma_a[u_{s,p} +h ] $ coincides with $\widetilde \Gamma_{a,s,p}[h]$ in the case where $s\in \R$ and $p=0$. 
\end{rem}
Now we have the following

\begin{prop}\label{step1}
Let $h\in C^{2,\al}(\overline\Omega)$ be small enough. Fix $x\in \partial\Omega$. For each $(s,p) \in \mathcal R\times \mathcal A(r_0)$, the complex-valued function given by $~(s,p)\xmapsto{\quad} \widetilde \Gamma_{a,s,p}[h](x)$ is the unique analytic continuation of $\widetilde \Gamma_{a,s,0}[h]$, $s\in \R$, to the complex-valued subset  $\mathcal R \times \mathcal A(r_0)\backslash\{0\} $ in $\Com^{n+1}$.
\end{prop}

\begin{proof}
Fix $h$ sufficiently small such that $\widetilde \Gamma_{a,s,p}[h]$ is well-defined for $s\in \mathcal R$ and $p\in \mathcal A(r_0)$. First note that, for each $x\in \partial\Omega$ fixed, the several complex-valued function 
\[
\mathcal R \times \mathcal A(r_0)  \ni (s,p) \xmapsto{\quad} \widetilde \Gamma_{a,s,p}[h](x) \in \Com
\]
is complex-valued analytic. This is just a consequence of the analytic character of the DN map \cite{Cal} with respect to the conductivity (see \eqref{DN_complex_1}), and composition arguments. (Recall that a several complex-valued function is analytic if on each coordinate it is itself a complex-valued analytic function.) 




\medskip

Consequently, given another analytic continuation of $\widetilde \Gamma_{a,s,0}[h]$ to the set $\mathcal R \times \mathcal A(r_0) \backslash\{0\}$, and since this last set is an analytic manifold (Lemma \ref{Aux_lemma}) of codimension 2, we conclude that both continuations must coincide for $p\in \mathcal A(r_0) \backslash\{0\}$ \cite[Theorem 2, p. 30]{Gunning}. The proof is complete.
%
%
\end{proof}

\subsection{Proof of Theorem \ref{MT}} We claim that Theorem \ref{MT} is a simple consequence of the following 

\begin{prop}[Reduction of the proof]\label{Reduction}
Under the assumptions of Theorem \ref{MT}, if  $\Gamma_{a_1}[ \tilde h] = \Gamma_{a_2}[\tilde h]$ for real-valued boundary valued functions $\tilde h\in C^{2,\al}(\overline{\Omega})$, then
\[
\Gamma_{a_1}[u_{s,p} + h]  = \Gamma_{a_2}[u_{s,p} + h],   \qquad (\hbox{cf. }~ \eqref{Extension_def})
\]
for any $s\in \R$ and $p\in \mathcal A(r_0)$, and for each $h\in C^{2,\al}(\overline\Omega)$ with sufficiently small norm.
\end{prop}

\begin{proof}
We must show that for all $h\in C^{2,\al}(\overline\Omega)$ with sufficiently small norm,
\[
\widetilde \Gamma_{a_1,s,p}[h]= \widetilde \Gamma_{a_2,s,p}[h].
\]
Since  $\Gamma_{a_1}[ \tilde h] = \Gamma_{a_2}[\tilde h]$, we have $\Gamma_{a_1}[ u_{s,0} + h] = \Gamma_{a_2}[u_{s,0} + h]$. From Lemma \ref{step_0} we have for $h$ small, and all $s\in \R$,
\[
\widetilde \Gamma_{a_1,s,0}[h]= \widetilde \Gamma_{a_2,s,0}[h].
\]
Hence, we conclude thanks to Proposition \ref{step1}.
\end{proof}

\begin{proof}[Proof of Theorem \ref{MT}]
From Proposition \ref{Reduction}, we have $\Gamma_{a_1}[u_{s,p} + h]  = \Gamma_{a_2}[u_{s,p} + h]$, and from Definition \ref{Extension}, this means that $ \widetilde \Gamma_{a_1,s,p}[h] = \widetilde \Gamma_{a_2,s,p}[h].$ Hence, using Corollary \ref{Regularity_3}, Lemma \ref{5p1}, Corollary \ref{Coincide_0} and Corollary \ref{5p3}, we have $ \widetilde \Gamma_{\ell,a_1,s,p}[h] = \widetilde \Gamma_{\ell,a_2,s,p}[h].$  The final conclusion comes from Theorem \ref{Uniqueness_linear}.
\end{proof}

\bigskip

%
%
%

\end{document}